\documentclass[1 leqno,11pt]{amsart}
\usepackage{amssymb, amsmath,latexsym,amsfonts,amsbsy, amsthm,mathtools,graphicx,,color}
\usepackage{float}
\usepackage{hyperref}

\numberwithin{equation}{section}

\allowdisplaybreaks

\let\al=\alpha

\let\d=\delta
\let\e=\varepsilon

\let\la=\lambda

\let\f=\frac

\let\na=\nabla
\let\th=\theta
\let\pa=\partial

\newcommand{\beq}{\begin{equation}}
\newcommand{\eeq}{\end{equation}}
\newcommand{\ben}{\begin{eqnarray}}
\newcommand{\een}{\end{eqnarray}}
\newcommand{\beno}{\begin{eqnarray*}}
\newcommand{\eeno}{\end{eqnarray*}}

\newtheorem{theorem}{Theorem}[section]

\newtheorem{lemma}[theorem]{Lemma}
\newtheorem{proposition}[theorem]{Proposition}

\newtheorem{remark}[theorem]{Remark}

\begin{document}

\title{} 
\title[instability of the unbounded shear layer]{Long-Wave Spectral Instability of  Shear Layers for the Compressible Euler Equations}

\author[C. Wang]{Chao Wang}
\address{School of Mathematical Sciences\\ Peking University\\ Beijing 100871, China}
\email{wangchao@math.pku.edu.cn}

\author[Y. Wang]{Yuxi Wang}
\address{School of Mathematics, Sichuan University, Chengdu 610064, China}
\email{wangyuxi@scu.edu.cn}

\author[W. Wu]{Wenzhi Wu}
\address{School of Mathematical Sciences\\ Peking University\\ Beijing 100871, China}
\email{wuwenzhi@stu.pku.edu.cn}

\author[Z. Zhang]{Zhifei Zhang}
\address{School of Mathematical Sciences, Peking University, Beijing 100871,  China}
\email{zfzhang@math.pku.edu.cn}
\maketitle

\begin{abstract}
We study the long-wave spectral instability of the two-dimensional compressible Euler equations around smooth monotone shear layers. We construct decaying half-line solutions of the compressible Rayleigh equation through a long-wave expansion and derive a second-order expansion of the matching Wronskian. For every fixed Mach number $m>0$, we prove the existence of unstable modes  for  sufficiently small wavenumbers. For $m<\sqrt2$, this holds for a  class of profiles, with $c_i$ tending to a positive constant as $\alpha\to0$. At $m=\sqrt{2}$, the profile $U_s(Y)=\tanh Y$ admits an unstable mode with $c\to0$ and $c_i$ of order $\alpha^{1/3}$. For $m>\sqrt2$, the same profile remains unstable, with $c\to c_*(m)\in(0,1)$ and $c_i>0$ of order $\alpha$. The corresponding temporal growth rates are of order $\alpha$, $\alpha^{4/3}$ and $\alpha^2$, respectively, showing a change in the long-wave instability scaling at $m=\sqrt2$.
In the zero-thickness limit, the supercritical unstable eigenvalue approaches the real axis, consistently with the stability results for supersonic compressible vortex sheets in \cite{CS1,CS2}.

\end{abstract}
\section{Introduction}

  This paper investigates the instability of a two-dimensional compressible flow for all Mach numbers  $m>0 $. The motion of the fluid is described by the dimensionless compressible Euler equations in the domain $(X,Y)\in \mathbb{R}\times\mathbb{R}$:
\begin{align}\label{eq:NS}
\left\{
\begin{aligned}
&\pa_t\rho^e+\na\cdot(\rho^e U^e)=0,\\
&\rho^e\big(\pa_tU^e+U^e\cdot \na U^e\big)+\na P(\rho^e)=0,\\
&\lim_{Y\to\pm \infty}U_2^e(Y)=0.
\end{aligned}
\right.
\end{align}
Here $\rho^e$ is the density, $U^e=(U^e_1, U_2^e)$ is the velocity, and  $P(\rho^e)$ is the pressure. We consider a  shear-layer profile connecting two uniform states at infinity, which means that $(\rho_s^e,U_s^e)(Y)=(1,U_s(Y),0)$ satisfies
\beno
\lim_{Y\to \pm \infty}U_s(Y)=\pm 1.
\eeno
It is easy to check that $(1,U_s(Y),0)$ is a steady solution to \eqref{eq:NS}.
To understand the (in)stability properties of the above shear layer  profile, we study the compressible Euler equations \eqref{eq:NS} linearized around $(1,U_s(Y),0)$. The linearized system on the whole domain $(X,Y)\in \mathbb{R}\times\mathbb{R}$ is given by
\begin{align}\label{eq:LCNS}
	\left\{
	\begin{aligned}
&\partial_t\varrho+U_s\partial_X\varrho+\nabla\cdot U=0,\\
		&\partial_tU+U_s\partial_X U+\partial_Y U_s(Y)(U_2,0)+m^{-2}\nabla \varrho=0,\\
		&\lim_{Y\to \pm \infty} U_2(Y)=0.
	\end{aligned}
	\right.
\end{align}
Here $\varrho =\rho^e-1$, $U=(U_1, U_2)= U^e -(U_s(Y),0)$, and $m=\frac{1}{\sqrt{P'(1)}}$ is the Mach number.

Our goal is to seek unstable solutions of the linearized system $\eqref{eq:LCNS}$ in the form 
\begin{align*}
	(\varrho,U_1, U_2)(t,X,Y)=(\rho,u,v)(Y)e^{\mathrm i\alpha(X-ct)}+\mathrm c.\mathrm c, 
\end{align*}
which means that $(\rho(Y),u(Y),v(Y))$ satisfies the following system
\begin{align}\label{eq:LCNS-Y}
	\left\{
	\begin{aligned}
		&\mathrm i\alpha(U_s-c)\rho+\mathrm{div}_\alpha(u,v)=0,\\
		&\mathrm i\alpha(U_s-c)u+v\partial_Y U_s+\mathrm i\alpha m^{-2}\rho=0,\\
		&\mathrm i\alpha(U_s-c)v+m^{-2}\partial_Y\rho=0,
	\end{aligned}
	\right.
\end{align}
where $\mathrm{div}_\alpha(u,v)=\mathrm i\alpha u+\partial_Yv$ and $v$ satisfies the boundary condition
\begin{align}\label{BC: u(0)=v(0)=0}
\lim_{Y\to\pm\infty}v(Y)=0.
\end{align}
 If for some $c\in\mathbb{C}$ with positive imaginary part $c_i > 0$ and wavenumber $\al>0$, the boundary value problem \eqref{eq:LCNS-Y} with \eqref{BC: u(0)=v(0)=0} has a non-trivial solution, then the shear layer profile $(1,U_s(Y),0)$ is spectrally unstable.  The goal of this paper is to construct an unstable solution to system \eqref{eq:LCNS-Y}-\eqref{BC: u(0)=v(0)=0} for all Mach numbers $m>0$.
 
By the first equation of \eqref{eq:LCNS-Y}, there exists a generalized stream function $\varphi(Y)$ such that 
\begin{align}\nonumber
	\partial_Y\varphi(Y)=u(Y)+\big(U_s(Y)-c\big)\rho(Y), \quad-i\alpha\varphi(Y)=v.
\end{align}
Then the system \eqref{eq:LCNS-Y} can be reduced to a Rayleigh-type system
\begin{align}\label{eq:GOS}
	\left\{
	\begin{aligned}
		&(U_s-c)\Lambda\varphi-\partial_Y(A^{-1}\partial_Y U_s)\varphi=0,\\
		&\rho=m^2A^{-1}(Y)\big(\partial_Y U_s\varphi-(U_s-c)\partial_Y\varphi\big),
	\end{aligned}
	\right.
\end{align}
where 
\begin{align*}
\Lambda=\partial_Y(A^{-1}\partial_Y)-\alpha^2,\quad A(Y)=1-m^2(U_s(Y)-c)^2,
\end{align*}
and the boundary condition \eqref{BC: u(0)=v(0)=0} is reduced to
\begin{align}\label{BC}
\varphi(\pm\infty)=0.
\end{align}
From system \eqref{eq:GOS}, we know that once we construct $\varphi$, we get the solution to $\rho$ according to the second equation in \eqref{eq:GOS}. For this reason, we only need to construct an unstable solution to $\varphi$ with the boundary condition \eqref{BC}.
\medskip

We begin by reviewing some results on the (in)stability of inviscid incompressible flows. Rayleigh \cite{Rayleigh} first established a necessary condition for instability: the presence of an inflection point in the shear flow. Fjørtoft \cite{Fjortoft} provided a refinement of Rayleigh’s condition. Lin established a sufficient instability criterion for a class of shear profiles \cite{Lin 2003}. His proof used a shooting argument for the Rayleigh equation. We refer the reader to \cite{FSV-1, FSV-2, Grenier, Lin-2004} for results on nonlinear instability.
\medskip

As for the (in)stability of inviscid compressible flows, we refer the reader to the instability result of Lees–Lin \cite{LL}, where Rayleigh's criterion for incompressible flows was extended to inviscid subsonic flows. Coulombel and Secchi established the linear stability of two-dimensional supersonic compressible vortex sheets for $m>\sqrt2$ in \cite{CS1} and subsequently proved their nonlinear stability by a Nash--Moser iteration in \cite{CS2}. They also analyzed the transition at the critical Mach number $m=\sqrt2$ in \cite{CS}.
 \medskip

We briefly review recent progress in the (in)stability theory of boundary layer flows. For shear profiles that are spectrally unstable at the inviscid level, Grenier and Nguyen \cite{GN} showed that the small viscosity does not suppress the instability, and the flow remains unstable at the viscous level. In contrast, for monotonic concave shear profiles, which are linearly stable in the inviscid setting, Grenier, Guo, and Nguyen \cite{GGN-DMJ} rigorously constructed unstable temporal modes by introducing the Rayleigh–Airy iterative scheme. These modes, referred to as Tollmien–Schlichting (T–S) waves, play an important role in the early stages of boundary layer transition.
The theoretical foundation of T–S waves originates from the pioneering works of Heisenberg, Tollmien, Schlichting, and C.~C.~Lin, among others \cite{DR, Lin, Sch}. The mathematical existence of neutral curves for such modes was established by Chen, Wu, and Zhang \cite{CWZ}; see also \cite{BG-2023} for related results. The construction of T–S waves has been extended to incompressible magnetohydrodynamic (MHD) flows in \cite{LYZ}.
For a broader discussion of nonlinear stability and instability in boundary layer theory, we refer to \cite{GMM-duke, GMM-arxiv, GN-arxiv, GM, GZ, CWZ1, CWZ2, GN} and references therein. Results concerning T–S waves in subsonic compressible flows can be found in \cite{YZ, MWWZ}, while the study
of Mack modes in supersonic boundary
layers is presented in \cite{MWWZ1}.
\medskip

The results above concern either discontinuous vortex sheets or wall-bounded shear flows. The long-wave spectral stability of smooth compressible shear layers in the whole plane, especially at and above the critical Mach number $m=\sqrt2$, is less understood. In this paper, we show that monotone shear layers admit unstable long-wave modes in all Mach numbers $m>0$. The critical and supercritical instabilities are weaker and are determined by a profile-dependent correction to the leading dispersion relation.

\medskip

Throughout this paper, we impose the following structural assumptions on the shear layer profile $U_s(Y)\in C^2(\mathbb{R})$: we assume that  there exists a constant $\theta>0$ such that
\begin{align}\label{eq:S-A}
\begin{split}
	&U_s(0)=0,\quad\lim_{Y\to\pm\infty}U_s(Y)=\pm 1,\quad\partial_YU_s(Y)>0,\\
	&\sup_{Y\geq0}|e^{\theta Y}\partial_Y^k(U_s(Y)-1)|<+\infty,\quad k=0,1,2,\\
	&\sup_{Y\leq0}|e^{-\theta Y}\partial_Y^k(U_s(Y)+1)|<+\infty,\quad k=0,1,2.
\end{split}
\end{align}
In particular, $U_s(Y)=\tanh Y$ satisfies the assumptions in \eqref{eq:S-A}.
\medskip

Now, we state our main result.  
 \begin{theorem}\label{thm:main1}
 Let $m>0$. Suppose that $U_s(Y)$ satisfies the structural assumptions \eqref{eq:S-A} when $m<\sqrt2$, and take $U_s(Y)=\tanh Y$ when $m\geq\sqrt2$.
Then there exists
a small constant $\alpha_0>0$ such that,
for every $0<\alpha<\alpha_0$, we can find $c(\al)$ such that there exists a non-trivial solution $(\varrho,U_1,U_2)(t,X,Y)$ to the system \eqref{eq:LCNS} of the form 
\begin{align*}
(\varrho,U_1,U_2)(t,X,Y)
=
(\rho,u,v)(Y)
e^{\mathrm{i}\alpha(X-ct)}
+\mathrm{c.c.},
\end{align*}
where $(\rho,u,v)
\in
W^{1,\infty}(\mathbb R)
\times W^{1,\infty}(\mathbb R)
\times W^{2,\infty}(\mathbb R)$ and solves the system \eqref{eq:LCNS-Y} subject to the boundary condition
\eqref{BC: u(0)=v(0)=0}.
Moreover, $\al$ and $c$ satisfy the following properties
\begin{enumerate}
\item For $0<m<\sqrt 2$, we have
\begin{align*}
c(\alpha)=i\sigma(m)+O(\alpha),
\end{align*}
where
\begin{align}\label{def: sigma(m)}
\sigma(m)=
\left(
\frac{\sqrt{4m^2+1}-(m^2+1)}{m^2}
\right)^{1/2}>0.
\end{align}
In particular,
\begin{align*}
c_r(\alpha)=O(\alpha),
\qquad
c_i(\alpha)=\sigma(m)+O(\alpha)>0,
\end{align*}

\item For $m=\sqrt{2}$, we have
\begin{align*}
c(\alpha)
=
e^{\pi\mathrm{i}/6}
\left(\frac{\alpha}{3}\right)^{1/3}
+
o(\al^\f13).
\end{align*}
In particular,
\begin{align*}
c_r(\alpha)
&=
\frac{\sqrt{3}}{2}
\left(\frac{\alpha}{3}\right)^{1/3}
+
o(\al^\f13),\quad 
c_i(\alpha)
=
\frac{1}{2}
\left(\frac{\alpha}{3}\right)^{1/3}
+
o(\al^\f13)>0.
\end{align*}

\item For $m>\sqrt{2}$, we have
\begin{align*}
c(\alpha)
=
c_*(m)+K_m\alpha
+
o(\al),
\end{align*}
where
\begin{align}
&c_*:=c_*(m)
=
\left(
\frac{m^2+1-\sqrt{4m^2+1}}{m^2}
\right)^{1/2}
\in(0,1),\label{c*}\\
&K_m
=
\frac{1-c_*^2}
{8c_*^2\sqrt{4m^2+1}}
\left[
2\pi c_*
+
i
\left(
4+2c_*
\log\frac{1+c_*}{1-c_*}
\right)
\right],\label{def: Km}\\
\nonumber
&\mbox{with}\quad \operatorname{Re}K_m>0,
\quad
\operatorname{Im}K_m>0.
\end{align}
In particular,
\begin{align*}
c_r(\alpha)
&=
c_*(m)+\operatorname{Re}K_m\,\alpha
+
o(\al),\quad c_i(\alpha)
=
\operatorname{Im}K_m\,\alpha
+
o(\al)>0.
\end{align*}
\end{enumerate}
\end{theorem}

 \begin{remark}
The expression defining $\sigma(m)$ in (1) has a removable singularity at
$m=0$. Indeed, $
\lim_{m\to0}\sigma(m)=1.
$
We therefore define $\sigma(0)=1$ by continuity. Thus, the value of $\sigma(0)$ is
understood in this limiting sense, which corresponds to the incompressible limit of the problem.
\end{remark}

\begin{remark}
  The regime of $ c $  indicates that the solution of the system \eqref{eq:LCNS} grows in time like
\begin{align*}
|(\varrho,U_1, U_2)(t,X,Y)| \sim e^{ \alpha c_i t} .
\end{align*}
Theorem \ref{thm:main1} shows that the growth rate $\alpha c_i$ is of order $\alpha$, $\alpha^{4/3}$, and $\alpha^2$ for $m<\sqrt2$, $m=\sqrt2$, and $m>\sqrt2$, respectively.
\end{remark}
\begin{remark}
Let $ U_k(Y) = U_s(k Y) $, where $ U_s(Y)=\tanh Y $. Thus, as $ k \to \infty $, 
$U_k'\rightharpoonup2\delta_0$ in the sense of distributions, where $\delta_0$ is the Dirac measure. 
If $W_k$ and $W$ denote the Wronskians associated with $U_k$ and $U_s$, respectively, then
\begin{align*}
W_k(\alpha,c)
=
kW\left(\frac{\alpha}{k},c\right).
\end{align*}
Thus, the reduced wavenumber is $\varepsilon=\alpha/k$. The
long-wave limit corresponds to fixed $k$ and $\alpha\to0$, whereas
the zero-thickness limit corresponds to fixed $\alpha$ and $k\to\infty$.
To connect with vortex sheets, we fix $\alpha>0$ and take $k\to\infty$, then the main theorem gives
\begin{align*}
c_i&\longrightarrow \sigma(m)>0,
&&m<\sqrt{2},\\
c_i&\sim
\frac12\left(\frac{\alpha}{3k}\right)^{1/3}\longrightarrow0,
&&m=\sqrt{2},\\
c_i&\sim
\operatorname{Im}K_m\,\frac{\alpha}{k}\longrightarrow0,
&&m>\sqrt{2}.
\end{align*}
Hence, the instability persists for $m<\sqrt{2}$, while it
degenerates at and above $m=\sqrt{2}$. The  case $m=\sqrt{2}$ complements
\cite{CS}, and the zero-thickness limit of  case $m>\sqrt{2}$  is consistent with the
stability results for supersonic compressible vortex sheets in
\cite{CS1,CS2}.
\end{remark}

\begin{remark}
For $m\geq\sqrt{2}$, we choose $U_s(Y)=\tanh Y$ to compute
$\mathcal I_{m,U_s}(c)$ explicitly. This is the profile-dependent
coefficient in the order-$\alpha^2$ term of the Wronskian expansion
\eqref{equ: dispersion}. For $m<\sqrt{2}$, the leading-order equation
$P_m(c)=0$ already has an unstable root for shear profiles satisfying
\eqref{eq:S-A}. At $m=\sqrt{2}$, the relevant leading-order root is
degenerate, while for $m>\sqrt{2}$ it is real. In these two cases, the sign of $\operatorname{Re}\mathcal I_{m,U_s}(c)$ determines whether the corrected root lies in the upper half-plane. More generally,
similar instability results hold for any admissible shear profile
satisfying
\begin{align*}
\lim_{\substack{c\to c_*+i0^+}}
\operatorname{Re}\mathcal I_{m,U_s}(c)<0.
\end{align*}
\end{remark}

\begin{remark}
Both the present modes and Mack modes are related to Rayleigh-type
equations in supersonic compressible flows. Mack modes rely on a
wall-acoustic mechanism in boundary layers \cite{MWWZ1}, whereas the
present long-wave modes arise by matching decaying solutions at infinity of a whole-space shear layer. Thus, they have different
instability mechanisms and frequency scalings.
\end{remark}

\begin{remark}

  The condition $ U_s(0) = 0 $ in \eqref{eq:S-A} is not necessary. Given that $ U_s(-\infty) = -1 $ and $ U_s(+\infty) = 1 $, we can find a point $Y_0 $ such that $ U_s(Y_0) = 0 $. We then define a new function $\tilde{U}_s(Y) = U_s(Y +Y_0)$, which satisfies the assumptions in \eqref{eq:S-A}.

 \end{remark}

{\bf Notations:}  
\begin{itemize}
\item Throughout this paper, $C>0$ denotes a generic constant independent of $\alpha$ and $c$ in the relevant parameter region. It may depend on the fixed Mach number $m$ and the profile $U_s$, and its value may change from line to line.

\item For two nonnegative quantities $X$ and $Y$, we write $X\lesssim Y$ if $X\leq CY$. The notation $X\gtrsim Y$ is defined similarly, and $X\approx Y$ means that both $X\lesssim Y$ and $X\gtrsim Y$ hold.

\item 
$
\mathbb G
=
\left\{
\alpha\in\mathbb R:0<\alpha\ll 1
\right\}.
$
\item 
$
\mathbb H_<
=
\left\{
(\alpha,c)\in\mathbb G\times\mathbb C:
\left|c-\mathrm{i}\sigma(m)\right|
\leq S\alpha
\right\},\quad m<\sqrt2.
$
\item
$\mathbb H_=
=
\left\{
(\alpha,c)\in\mathbb G\times\mathbb C:
\left|
c-e^{\pi\mathrm{i}/6}\delta
\right|
\leq
S\delta^2|\log\delta|^2
\right\},
\qquad
\delta
=
\left(\frac{\alpha}{3}\right)^{1/3},\quad m=\sqrt 2.
$

\item $\mathbb H_>
=
\left\{
(\alpha,c)\in\mathbb G\times\mathbb C:
\left|
c-\bigl(c_*(m)+K_m\alpha\bigr)
\right|
\leq
S\alpha^2|\log \al|^6
\right\},\quad m>\sqrt 2.$

\item $\mathbb{H}=
\begin{cases}
    \mathbb{H}_<,\quad 0<m<\sqrt{2},\\
     \mathbb{H}_=,\quad m=\sqrt{2},\\
       \mathbb{H}_>,\quad m>\sqrt{2}.
\end{cases}$

\end{itemize}
\section{Sketch of the proof}
\label{sec:proof-strategy}
The spectral problem is reduced to constructing an unstable solution of the compressible Rayleigh equation. More precisely, we seek a non-trivial function $\varphi$ and an eigenvalue $c$ with $c_i>0$ such that
\begin{equation}
\begin{cases}
(U_s-c)\Lambda\varphi
-\partial_Y(A^{-1}\partial_YU_s)\varphi=0,
& Y\in\mathbb R,\\
\varphi(\pm\infty)=0,
\end{cases}
\label{R11p}
\end{equation}
where
\begin{align*}
\Lambda
&=
\partial_Y(A^{-1}\partial_Y)-\alpha^2,
\quad 
A(Y)=
1-m^2(U_s(Y)-c)^2.
\end{align*}

Instead of solving this problem directly on the whole line, we
construct non-trivial solutions $\varphi_+$ and $\varphi_-$ that decay
at $+\infty$ and $-\infty$, respectively. They solve the half-line
problems
\begin{equation}
\begin{cases}
(U_s-c)\Lambda\varphi_\pm
-\partial_Y(A^{-1}\partial_YU_s)\varphi_\pm=0,
& Y\in\mathbb R_\pm,\\
\varphi_\pm(\pm\infty)=0.
\end{cases}
\label{R16}
\end{equation}

The main  task is to construct these decaying half-line
solutions uniformly in the long-wave regime and to obtain accurate
expansions at $Y=0$. Once $\varphi_\pm$ are
constructed, they give rise to a non-trivial global solution of
\eqref{R11p} if and only if their Wronskian vanishes at $Y=0$, namely,
\begin{equation}
W(\alpha,c)
\coloneqq
\begin{vmatrix}
\varphi_+(0) & \varphi_-(0)\\
\partial_Y\varphi_+(0) & \partial_Y\varphi_-(0)
\end{vmatrix}
=0.
\label{D11}
\end{equation}
This matching condition \eqref{D11} is referred to as the dispersion relation.

Therefore, the remainder of the
proof is divided into two main parts: the construction and
estimation of the decaying half-line solutions, and the analysis of
the zeros of $W(\alpha,c)$ in the three Mach-number ranges.

\subsection{Decaying half-line solutions}
To construct the solution to the Rayleigh equation  \eqref{R16}, we set
\begin{align*}
A_{\pm\infty}
=
1-m^2(1\mp c)^2,
\quad
\gamma_\pm
=
A_{\pm\infty}^{1/2},
\quad
\beta_\pm
=
\alpha\gamma_\pm,
\end{align*}
where the square-root branches are chosen so that $\operatorname{Re}\beta_\pm>0.$

To extract the exact far-field factors, we seek the
half-line solutions in the form
\begin{align}\label{def: varphi}
\varphi_\pm(Y)
&=
e^{\mp\beta_\pm Y}\chi_\pm(Y),\quad
\chi_\pm
=
\sum_{n=0}^{\infty}
(\pm\beta_\pm)^n\chi_{\pm,n},
\quad
\chi_{\pm,0}=U_s-c,
\end{align}
Here, $\chi_{\pm,n}$ satisfies the equation $Ray_0[\chi_{\pm,n}]=F_{\pm,n}$, and  $F_{\pm, n+1}$ contains some terms composed of $\chi_{\pm, n-1}$ and $\pa_Y\chi_{\pm, n-1}$, and see \eqref{CR4}-\eqref{def: F_n+2} for more details.

When $m<\sqrt{2}$ and $c$ is close to
$\mathrm{i}\sigma(m)$, the factors $U_s-c$ and $A(Y)$ remain uniformly
separated from zero. At and above the critical Mach number, the factors $U_s-c$ and $A(Y)$
 become small near critical points
\begin{align*}
U_s=c_r,
\qquad
U_s-c_r=\pm\frac{1}{m}.
\end{align*}
The apparent algebraic singularities are reduced to logarithmic
losses by the divergence structure of $Ray_0$. The series \eqref{def: varphi} is well-defined due to
\begin{align*}
\alpha |\log c_i|^2\ll1,
\end{align*}
 holds in all the parameter regimes considered below. 
 
\subsection{The expansion of the dispersion relation}

Expanding $\varphi_\pm(0)$ and $\partial_Y\varphi_\pm(0)$ to second
order in $\alpha$ gives a dispersion relation of the form
\begin{align*}
(1-m^2c^2)^{-1}W(\alpha,c)
=
\frac{4\alpha cP_m(c)}{D_m(c)}
+
\alpha^2\mathcal I_{m,U_s}(c)
+
\mathcal R(\alpha,c),
\end{align*}
where the precise definitions of $P_m(c)$, $D_m(c)$ and $\mathcal I_{m,U_s}(c)$ are given in \eqref{def: Pm, Dm} and \eqref{def: ImU_s}.

The denominator $D_m(c)$ stays away from zero in the regions under
consideration. The first term is determined entirely by the two end
states of the shear layer. In particular, it is independent of the
detailed shape of $U_s$. The second-order coefficient
$\mathcal I_{m,U_s}$ contains the profile-dependent contribution.

\subsection{Solve the dispersion relation in the three Mach-number regimes}

We  describe how the three asymptotic laws follow from the root
structure of $P_m$.
\begin{enumerate}
\item {\underline {The case $m<\sqrt{2}$.}}
The polynomial $P_m$ has a simple root in the upper half-plane,
\begin{align*}
c=\mathrm{i}\sigma(m),
\end{align*}
where $\sigma(m)$ is given in \eqref{def: sigma(m)}.

Since this root is simple and the second-order term remains bounded,
the dispersion relation can be solved in a neighborhood of
$\mathrm{i}\sigma(m)$. A contraction argument gives
\begin{align*}
c(\alpha)
=
\mathrm{i}\sigma(m)+O(\alpha).
\end{align*}
Thus, $c_i$ remains bounded away from zero as $\alpha\to0$.

\item{\underline{The case $m=\sqrt{2}$.}}
At the critical Mach number, the relevant root reaches the origin.
More precisely,
\begin{align*}
P_{\sqrt{2}}(c)
=
-6c^2+O(c^4).
\end{align*}
For $U_s(Y)=\tanh Y$, the second-order coefficient satisfies
\begin{align*}
\mathcal I_{\sqrt{2},U_s}(c)
=
-4+O\bigl(|c||\log c_i|\bigr).
\end{align*}
The leading balance is therefore
\begin{align*}
c^3
=
\frac{\mathrm{i}\alpha}{3}
+
\text{higher-order terms}.
\end{align*}
Selecting the root in the upper half-plane yields
\begin{align*}
c(\alpha)
=
e^{\pi\mathrm{i}/6}
\left(\frac{\alpha}{3}\right)^{1/3}
+
o\left(\alpha^{1/3}\right).
\end{align*}

\item{\underline{The case $m>\sqrt{2}$.}}
In this case, $P_m$ has a simple real root $c_*(m)$ given in \eqref{c*}.
Therefore, the leading-order relation  gives a neutral real phase
velocity. The profile-dependent second-order term determines whether
this root moves into the upper or lower half-plane.
For $U_s(Y)=\tanh Y$, the integral $\mathcal I_{m,U_s}(c)$ has the limiting value $-\xi_m$, where
\begin{align*}
\xi_m
=
4
+
2c_*
\left(
\log\frac{1+c_*}{1-c_*}
-
\pi i
\right).
\end{align*}
It determines the sign of the imaginary part of the first correction. Balancing the linear
variation of the leading term at $c=c_*$ with the second-order term
gives
\begin{align*}
c(\alpha)
=
c_*(m)+K_m\alpha+o(\alpha),\quad \operatorname{Im}K_m>0,
\end{align*}
where $K_m$ is given in \eqref{def: Km}.
Thus, the  correction moves the real leading root
into the upper half-plane and produces a weak long-wave instability.
\end{enumerate}

In each case, the approximate root determines the natural rescaling
of $c$. We use $c-\mathrm{i}\sigma=O(\alpha)$ for
$m<\sqrt{2}$, $c=O(\alpha^{1/3})$ for $m=\sqrt{2}$, and
$c-c_*=O(\alpha)$ for $m>\sqrt{2}$. The remainder estimates
yield a contraction mapping in the corresponding complex
neighborhood and hence an exact zero of $W(\alpha,c)$.

The half-line construction is given in Section~3, the Wronskian
expansion is derived in Section~4, the required uniform estimates are
proved in Section~5, and the three dispersion relations are solved
in Section~6.

\section{The compressible Rayleigh equation}
In this section, we first construct the solution $\varphi_{\pm}(Y), Y\in \mathbb{R}_{\pm}$ of the compressible Rayleigh equation:
\begin{align}\label{CR50}
\left\{
\begin{aligned}
&(U_s-c)\Lambda\varphi_\pm-\partial_Y(A^{-1}\partial_YU_s)\varphi_\pm=0,\ Y\in \mathbb{R}_\pm ,\\
   &  \varphi_{\pm}(\pm\infty) =0,
\end{aligned}
\right.
\end{align}
where
\begin{align*}
\Lambda=\partial_Y(A^{-1}\partial_Y)-\alpha^2,\quad A(Y)=1-m^2(U_s(Y)-c)^2.
\end{align*}
We introduce
\begin{align}\label{def: A_pm}
 A_{\pm\infty}:=\lim_{Y\to \pm\infty} A(Y)=1-m^2(1\mp c)^2.
\end{align}
If $(\alpha,c)\in \mathbb{H}$, we know $\operatorname{Im} (1\mp c)^2 \in \mathbb{R}_\mp$, which implies $\operatorname{Im}  A_{\pm\infty} \in \mathbb{R}_\pm$. Moreover, we define
\begin{align}
\beta_\pm=\gamma_\pm \alpha,\quad \gamma_\pm =A_{\pm\infty}^\frac{1}{2},\quad \text{with}\quad \arg(\gamma_+)\in (0,\frac{\pi}{2}) ,\quad \arg(\gamma_-)\in (-\frac{\pi}{2},0) . 
\end{align}
 

 \medskip

First, we give some facts that are used frequently.
\begin{lemma}

\begin{enumerate}
\item For $0<m<\sqrt 2$ and $(\alpha,c)\in \mathbb{H}_<$, $|U_s(Y)-c|$ and $|A(Y)|$ have upper and lower bounds,  i.e.,  there exists $C>1$ such that
\begin{align}\label{est1: U_s-c, A}
 0<C^{-1}\leq |U_s(Y)-c| \leq C, \quad 0<C^{-1}\leq |A(Y)|\leq C\quad \text{ for any } Y\in \mathbb{R}.
\end{align}
\item For $m\geq \sqrt{2}$ and $(\alpha,c)\in \mathbb{H}_\geq$,  there exists $C>1$ such that
\begin{align}\label{est2: U_s-c, A}
 0<C^{-1}c_i\leq |U_s(Y)-c| \leq C, \quad 0<C^{-1}c_i\leq |A(Y)|\leq C\quad \text{ for any } Y\in \mathbb{R}.
\end{align}
\item For $m>0$ and $(\alpha,c)\in \mathbb{H}$, $|\gamma_\pm|$ have upper and lower bounds,  i.e.,  there exists $C>1$ such that
\begin{align}\label{est: gamma}
  0<C^{-1}\leq |\gamma_\pm|\leq C .
\end{align}
\end{enumerate}
\end{lemma}

\medskip

To simplify the notation, we introduce two operators:
\begin{align*}
&Ray[\varphi]\coloneqq (U_s-c)\Lambda\varphi-\partial_Y(A^{-1}\partial_YU_s)\varphi,\\
&{Ray}_0[\varphi]\coloneqq (U_s-c)\partial_Y(A^{-1}\partial_Y\varphi)-\partial_Y(A^{-1}\partial_YU_s)\varphi,
\end{align*}
which have the following relationship 
\[
Ray[\varphi]={Ray}_0[\varphi]-\al^2(U_s-c)\varphi.
\]
Moreover,  the operator ${Ray}_0[\cdot]$ has the following equivalent formula:
\begin{equation}
\begin{aligned}
{Ray}_0[\varphi]= \partial_Y \Big( 
\frac{{(U_s-c)}^2}{A(Y)}
\partial_Y  ( 
\frac{\varphi}{U_s-c} )
\Big).
\label{CR11}
\end{aligned}
\end{equation}
We look for a solution ${\varphi_{\pm}}$ to system \eqref{CR50} with the following formula:
\begin{equation}
\begin{aligned}
\varphi_\pm (Y)=e^{\mp \beta_\pm Y}\chi_\pm (Y).
\label{R21}
\end{aligned}
\end{equation}
Thus, by direct calculation, we obtain
\begin{equation}
\begin{aligned}
(\Lambda+\alpha^2)\varphi_\pm &=\partial_Y \big( 
A^{-1} e^{\mp \beta_\pm Y}
\big( \mp \beta \chi_\pm
+\partial_Y \chi_\pm
\big)
\big)\\
&=\partial_Y \big( 
 e^{\mp \beta_\pm Y}
\big( \mp \beta A^{-1}\chi_\pm
+ A^{-1}\partial_Y\chi_\pm
\big)
\big)\\
&= e^{\mp \beta_\pm Y}\big(
\mp \beta_\pm
\big( \mp \beta A^{-1}\chi_\pm
+ A^{-1}\partial_Y\chi_\pm
\big)
+\partial_Y \big( \mp \beta A^{-1}\chi_\pm
+ A^{-1}\partial_Y\chi_\pm
\big)
\big)\\
&= e^{\mp \beta_\pm Y}\big(
\beta_\pm^2 A^{-1}\chi_\pm
\mp \beta_\pm \big( 
2A^{-1}\partial_Y\chi_\pm+\partial_Y(A^{-1})\chi_\pm
\big)
+(\Lambda+\alpha^2)\chi_\pm 
\big).
\label{RR21}
\end{aligned}
\end{equation}
Bring \eqref{R21} and \eqref{RR21} into \eqref{CR50} to get the equation of $\chi_\pm$:
\begin{align*}
Ray_0[\chi_{\pm}]=	
\pm \beta_\pm (U_s-c)\big( 2A^{-1}\partial_Y\chi_\pm+\partial_Y(A^{-1})\chi_\pm\big) - \beta_\pm^2 (U_s-c)\big(A^{-1}\chi_\pm \big)
+\alpha^2 (U_s-c)\chi_\pm. 
\end{align*}

To solve for $\chi_\pm(Y)$, we seek a solution of the following form
\begin{equation}
\begin{aligned}
\chi_{\pm}(Y)=\sum_{n=0}^{\infty} {(\pm  \beta_\pm)}^n \chi_{\pm,n}(Y).
\label{CR3}
\end{aligned}
\end{equation}

Thus, we seek  $\chi_{\pm,n}(Y) ,n\geq 0$ satisfying
\begin{equation}
\begin{aligned}
\begin{cases}
    {Ray}_0[\chi_{\pm,0}]=0,\\
    {Ray}_0[\chi_{\pm,1}]=F_{\pm,1},\\
 {Ray}_0[\chi_{\pm,n+2}]
= F_{\pm,n+2},\quad n \geq 0,
\end{cases}
\label{CR4}
\end{aligned}
\end{equation}
with the boundary conditions 
\begin{align}\label{BC: chi_n}
\chi_{\pm,n}(\pm \infty)=\pa_Y\chi_{\pm,n}(\pm \infty)= 0,\quad n\geq 1,
\end{align}
where
\begin{align}
&F_{\pm,1}=(U_s-c)\big( 2A^{-1}\partial_Y\chi_{\pm,0}+\partial_Y(A^{-1})\chi_{\pm,0}\big),\label{def: F_1}\\
&F_{\pm,n+2}=(U_s-c)\big( 2A^{-1}\partial_Y\chi_{\pm,n+1}+\partial_Y(A^{-1})\chi_{\pm,n+1}\big)\label{def: F_n+2}\\
\nonumber
&\qquad\qquad- (U_s-c)\big(A^{-1}-\gamma_\pm^{-2} \big)\chi_{\pm,n},\quad n\geq 0.
\end{align}

For $n=0$, it is easy to get $\chi_{\pm,0}=U_s-c$ as a solution to ${Ray}_0[\chi_{\pm,0}]=0$.
By \eqref{CR4} and \eqref{CR11}, we derive the expression for $\chi_{\pm,n}$.
Formally, we have constructed a non-trivial solution $\varphi_{\pm}\in C^2(\mathbb{R}_\pm)$ to the Rayleigh equation \eqref{CR50}. 
It remains to establish the uniform estimates for $\chi_{\pm,n}$ to make sure \eqref{CR3} is well-defined.

The proof of the convergence of \eqref{CR3} is given in Section 5, and we first present the result as follows.
\begin{proposition}\label{pro: varphi}
Let $m>0$ and $(\alpha,c)\in \mathbb{H}$. Then there exists a non-trivial solution $\varphi_{\pm}\in C^2(\mathbb{R}_\pm)$ to the Rayleigh equation \eqref{CR50}. Moreover,
\begin{equation}\label{con:HR}
\begin{aligned}
&\varphi_\pm (Y)=e^{\mp \beta_\pm Y}\chi_\pm (Y),\quad \chi_\pm (Y)=\sum_{n=0}^{\infty} {(\pm  \beta_\pm)}^n \chi_{\pm,n}(Y),\\
&\|\chi_\pm (Y)-(U_s(Y)-c)\|_{Y_{\th,\pm}}\lesssim \al\la,\quad \|\chi_{n,\pm}\|_{Y_{\th,\pm}}\leq (C \la)^n,\quad n\geq 1
\end{aligned}
\end{equation}
where $\lambda=1$ when $m<\sqrt{2}$, and $\lambda=|\log c_i|^2$ when $m\geq \sqrt{2}$.
Here the definition of the norm $\|\cdot\|_{Y_{\th}}$ is given in \eqref{Y_th}.
\end{proposition}

\section{Dispersion relation of the compressible Rayleigh equation}

The aim of this section is to construct the solution to  \eqref{R11p}. By now, we have constructed $\varphi_{\pm}$ by Proposition \ref{pro: varphi}. Define
\begin{align*}
\begin{cases}
  \varphi(Y)=C_{+}\varphi_{+}(Y),Y \geq 0,\\
\varphi(Y)=-C_{-}\varphi_{-}(Y),Y \leq 0.
\end{cases}
\end{align*}
We select appropriate non-zero constants $C_{+},C_{-}$ such that $\varphi(Y)$ becomes a solution to \eqref{R11p}. That means that $(C_{+},C_{-})$ needs to satisfy
\begin{align*}
\begin{pmatrix}
    \varphi_{+}(0)&   \varphi_{-}(0)\\
     \partial_Y \varphi_{+}(0)& \partial_Y\varphi_{-}(0)
\end{pmatrix}
\begin{pmatrix}
    C_{+}\\
  C_{-}
\end{pmatrix}=0.
\end{align*}
The existence of a non-trivial solution $(C_{+},C_{-})$, as well as the Rayleigh equation \eqref{R11p}, is guaranteed by the following dispersion relation:
\begin{align*}\
W(\alpha,c)\coloneqq 
\begin{vmatrix}
\varphi_{+}(0)&   \varphi_{-}(0)\\
     \partial_Y \varphi_{+}(0)& \partial_Y\varphi_{-}(0)
\end{vmatrix}
=0.
\end{align*}

In this section, we provide a rigorous formulation of $W(\alpha,c)$ and subsequently solve the dispersion relation $W(\alpha,c)=0$ in Section 6.

\begin{proposition}\label{con:disp}
  Let $m>0$ and suppose $\varphi_{\pm}$ is the non-trivial solution to \eqref{CR50} constructed in Section 3. For $(\alpha,c)\in \mathbb{H}$,  $W(\alpha,c)$ admits the following representation:
\begin{equation}\label{equ: dispersion}
\begin{aligned}
&	{(1-m^2c^2)}^{-1} W(\alpha,c)
=\frac{4\alpha cP_m(c)}{D_m(c)}
+
\alpha^2\mathcal I_{m,U_s}(c)+O(\alpha^3\lambda^3),
\end{aligned}
\end{equation}
where
\begin{align}
&P_m(c)
=
m^2(1-c^2)^2-2(1+c^2),\quad D_m(c)
=
\gamma_+\gamma_-
\left[
(1-c)^2\gamma_-
-
(1+c)^2\gamma_+
\right],\label{def: Pm, Dm}\\
&\mathcal I_{m,U_s}(c)=\int_0^{+ \infty} \frac{ ({(U_s-c)}^2-{(1- c)}^2 )
 ( 
{(U_s-c)}^2
-\frac{{(1+c)}^2}{\gamma_+ \gamma_-}
 )
}{{(U_s-c)}^2}
dZ\label{def: ImU_s}\\
\nonumber
&\quad\quad\quad\quad\quad +  \int_{- \infty}^0 \frac{({(U_s-c)}^2-{(1+ c)}^2 )
 ( 
{(U_s-c)}^2-
\frac{{(1-c)}^2}{\gamma_+ \gamma_-}
 )
}{{(U_s-c)}^2}
dZ, 
\end{align}
and
\begin{align*}
\lambda=\begin{cases}
  1,\quad m<\sqrt{2},\\
|\log c_i|^2, \quad m\geq \sqrt{2}.
\end{cases}
\end{align*}

\end{proposition}

\begin{lemma}
\label{dis 1}
 Let $m> 0$ and $(\alpha,c)\in \mathbb{H}$, we have $$|{(1- c)}^2 \gamma_--
 {(1+ c)}^2\gamma_+ |\geq C^{-1},$$
 for some constant $C>0$. More precisely, for $m \geq  \sqrt{2}$ and $(\alpha,c)\in \mathbb{H}_\geq$, the values $\gamma_\pm$   satisfy:
\begin{equation}
\begin{aligned}
\label{U3}
\gamma_{+} =& \frac{1-c_*}{1+c_*}i+O(|c-c_*|),\\
\gamma_{-} = & -\frac{1+c_*}{1-c_*}i+O(|c-c_*|),
\end{aligned}
\end{equation}
and
\begin{equation}
\begin{aligned}
\label{U2}
\gamma_+ \gamma_-=&1+O(|c-c_*|), \\
{(1- c)}^2 \gamma_--
 {(1+ c)}^2\gamma_+=&-2(1-c_*^2)i+O(|c-c_*|).
\end{aligned}
\end{equation}

\end{lemma}

\begin{proof}
For $m<\sqrt{2}$,  we notice $\sigma(m)\in (0,1]$ and $c$ is near $\sigma i$, so there exists $\theta_0>0 $, such that $$\arg(1- c)\in (-\frac{\pi}{4}+\theta_0,-\theta_0).$$
 Then we obtain
$$\arg\big({(1- c)}^2 \gamma_-\big)\in (-\pi+2\theta_0,-2\theta_0),$$
which implies$$\operatorname{Im}\big({(1- c)}^2 \gamma_-\big)\leq -|{(1- c)}^2 \gamma_-|\sin{2\theta_0}\leq -C^{-1}.$$
 Similarly, we have
$$\operatorname{Im}\big({(1+ c)}^2 \gamma_+\big)\geq C^{-1}.$$
Thus, we conclude $$|{(1- c)}^2 \gamma_--
 {(1+ c)}^2\gamma_+| \geq \big|\operatorname{Im}\big({(1- c)}^2 \gamma_--
 {(1+ c)}^2\gamma_+\big)\big|\geq C^{-1}.$$
 
For $m=\sqrt{2}$, we have
\begin{align*}
	\gamma_\pm=\sqrt{
 1-2{(1\mp c)}^2
    }
    =\sqrt{-1\pm 4c -2c^2}
    =\pm \mathrm{i}
    \sqrt{1\mp 4c +2c^2} =\pm i+O(c).
\end{align*}

For $m>\sqrt{2}$, we derive an explicit expression for $m^2$ in terms of $c_*^2$.  We notice
\begin{equation}
m^2 = \frac{2(1+c_*^2)}{(1-c_*^2)^2}. \label{eq:m_expr}
\end{equation}

Now, we evaluate $A_{+\infty}$:
\[ A_{+\infty} = 1 - m^2(1-c_*)^2+O(|c-c_*|). \]
Substituting the expression for $m^2$ from \eqref{eq:m_expr}:
\begin{align*}
A_{+\infty}= 1 - \frac{2(1+c_*^2)}{(1-c_*)^2(1+c_*)^2} (1-c_*)^2+O(|c-c_*|)= -\frac{(c_*-1)^2}{(1+c_*)^2}+O(|c-c_*|).
\end{align*}
Since $c_*\in (0,1)$, we have
\begin{align*}
\gamma_{+} = \frac{1-c_*}{1+c_*}i+O(|c-c_*|).
\end{align*}
Next, we evaluate $A_{-\infty}$:
\[ A_{-\infty} = 1 - m^2(-1-c_*)^2 = 1 - m^2(1+c_*)^2. \]
Substituting the expression for $m^2$ from \eqref{eq:m_expr}:
\begin{align*}
A_{-\infty} = 1 - \frac{2(1+c_*^2)}{(1-c_*)^2(1+c_*)^2} (1+c_*)^2 +O(|c-c_*|) = -\frac{(c_*+1)^2}{(1-c_*)^2}+O(|c-c_*|).
\end{align*}
Since $c_*\in (0,1)$, we have
\begin{align*}
\gamma_{-} =  -\frac{1+c_*}{1-c_*}i+O(|c-c_*|).
\end{align*}
Then we directly derive \eqref{U2} by \eqref{U3}.

\end{proof}

%

    \bigskip

\underline{{\it Proof of Proposition \ref{con:disp}}.} To get $W(\alpha,c)$, the main objective is to get the value of $\chi_{\pm,n} (0)$ and $\partial_Y\chi_{\pm,n} (0)$. 

$\bullet$  Firstly, recalling the definition of $\chi_{\pm,0}=U_s-c$, we get
$$\chi_{\pm,0}(0)=-c,\quad \partial_Y\chi_{\pm,0}(0)=\partial_YU_s(0).$$

$\bullet$  For $\chi_{\pm,1}$,  by \eqref{CR4} and \eqref{def: F_1}, we have
\begin{align*}
Ray_0\big[ \chi_{\pm,1}\big]	=&
2{A(Y)}^{-1}  (U_s-c)  \partial_Y(U_s-c)+{(U_s-c)}^2\partial_Y({A(Y)}^{-1})\\
=&\partial_Y 
( 
{A(Y)}^{-1} {(U_s-c)}^2
).
\end{align*}
Using the boundary condition \eqref{BC: chi_n}, we obtain
\begin{equation}
\begin{aligned}
\label{CR5}
	\chi_{\pm,1}(Y)&=(U_s-c) \int_{\pm \infty}^Y  \big(\frac{A(Z)}{{(U_s-c)}^2}  \big) \big(
 \int_{\pm \infty}^Z  
 \partial_{Z' }
\big( 
{A(Z')}^{-1} {(U_s-c)}^2
\big)
 dZ'
\big)dZ\\
 &=(U_s-c) \int_{\pm \infty}^Y 
 1-\frac{{(1\mp c)}^2 A(Z)}{{(U_s-c)}^2\gamma_\pm^2}
 dZ=(U_s-c) \int_{\pm \infty}^Y 
 \frac{{(U_s-c)}^2-{(1\mp c)}^2}{{(U_s-c)}^2\gamma_\pm^2}
 dZ.
\end{aligned}
\end{equation}
On the one hand, notice that
$${(U_s-c)}^2\gamma_\pm^2-{(1\mp c)}^2 A(Z)={(U_s-c)}^2-{(1\mp c)}^2.$$
Act $\pa_Y$ on both sides of \eqref{CR5} to get
\begin{equation}
\begin{aligned}
\partial_Y \chi_{\pm,1}(Y)
 &=\partial_Y U_s \int_{\pm \infty}^Y 
 \frac{{(U_s-c)}^2-{(1\mp c)}^2}{{(U_s-c)}^2\gamma_\pm^2}
 dZ
 +
 \frac{{(U_s-c)}^2-{(1\mp c)}^2}{{(U_s-c)}\gamma_\pm^2}.
 \label{CR6}
\end{aligned}
\end{equation}

Take $Y=0$ on both sides of \eqref{CR5} and \eqref{CR6} to obtain
\begin{align*}
	\chi_{\pm,1}(0)=-cB_\pm, \quad \partial_Y 	\chi_{\pm,1}(0)=\partial_Y U_s(0)B_\pm-cD_\pm, 
\end{align*}
where
\begin{align*}
B_\pm = \int_{\pm \infty}^0 
 \frac{{(U_s-c)}^2-{(1\mp c)}^2}{{(U_s-c)}^2\gamma_\pm^2}
 dZ,\quad D_\pm =  
 1-\frac{{(1\mp c)}^2 (1-m^2 c^2)}{c^2\gamma_\pm^2}.
\end{align*}

$\bullet$ For $\chi_{\pm,2}$,  take $n=0$ on \eqref{CR4} and \eqref{def: F_n+2} to get
\begin{align*}
Ray_0\big[ \chi_{\pm,2}\big]	=&
 (U_s-c) \big( 2{A(Y)}^{-1}\partial_Y\chi_{\pm,1}+\partial_Y\big({A(Y)}^{-1}\big)\chi_{\pm,1} \big)
\\
&- (U_s-c)\big({A(Y)}^{-1}-\gamma_\pm^{-2} \big)\chi_{\pm,0}.
\end{align*}
Use \eqref{CR5} and \eqref{CR6} to obtain 
\begin{align*}
	Ray_0\big[ \chi_{\pm,2}\big]	=&
 \big( 
 \frac{2(U_s-c)\partial_Y U_s}{A(Y)}+
 {(U_s-c)}^2 \partial_Y\big({A(Y)}^{-1}\big)
 \big)
 \int_{\pm \infty}^Y 
 1-\frac{{(1\mp c)}^2 A(Z)}{{(U_s-c)}^2\gamma_\pm^2}
 dZ\\
 &+\frac{2{(U_s-c)}^2}{A(Y)}
\big(
 1-\frac{{(1\mp c)}^2 A(Y)}{{(U_s-c)}^2\gamma_\pm^2}\big)
- {(U_s-c)}^2\big({A(Y)}^{-1}-\gamma_\pm^{-2} \big)\\
	=&
 \partial_Y 
\big( 
\frac{{(U_s-c)}^2}{A(Y)}
\big)
 \int_{\pm \infty}^Y 
 1-\frac{{(1\mp c)}^2 A(Z)}{{(U_s-c)}^2\gamma_\pm^2}
 dZ\\
 &+\frac{2{(U_s-c)}^2}{A(Y)}
\big(
 1-\frac{{(1\mp c)}^2 A(Y)}{{(U_s-c)}^2\gamma_\pm^2}\big)
- {(U_s-c)}^2\big({A(Y)}^{-1}-\gamma_\pm^{-2} \big)\\
\coloneqq & T_1(Y)+T_2(Y).
\end{align*}
Using the boundary condition \eqref{BC: chi_n} again,  we have
\begin{equation}
\chi_{\pm,2}(Y)=(U_s-c)\int_{\pm \infty}^Y G_\pm(Z)dZ,  \label{CR47}
\end{equation}
with 
\beno
G_\pm(Y)\coloneqq
 \frac{A(Y)}{{(U_s-c)}^2}\int_{\pm \infty}^Y T_1(Z)+T_2(Z)
dZ.
\eeno
Integration by parts to get
\begin{align*}
	\int_{\pm \infty}^Y T_1(Z)
dZ
=&\int_{\pm \infty}^Y \big( 
 \int_{\pm \infty}^Z 
 1-\frac{{(1\mp c)}^2 A(Z')}{{(U_s-c)}^2\gamma_\pm^2}
 dZ'
\big)d\big( 
\frac{{(U_s-c)}^2}{A(Z)}
\big)\\
=&
\frac{{(U_s-c)}^2}{A(Y)}
  \int_{\pm \infty}^Y \big(
 1-\frac{{(1\mp c)}^2 A(Z)}{{(U_s-c)}^2\gamma_\pm^2}\big)dZ
 -
 \int_{\pm \infty}^Y
\frac{{(U_s-c)}^2}{A(Z)}
\big( 
 1-\frac{{(1\mp c)}^2 A}{{(U_s-c)}^2\gamma_\pm^2}
\big)dZ,
\end{align*}
which implies that
\begin{align*}
	\int_{\pm \infty}^Y T_1(Z)+T_2(Z)
dZ
=&\frac{{(U_s-c)}^2}{A(Y)}
\int_{\pm \infty}^Y  \big(
 1-\frac{{(1\mp c)}^2 A(Z)}{{(U_s-c)}^2\gamma_\pm^2}\big)dZ\\
 &+ \int_{\pm \infty}^Y\big[
\frac{{(U_s-c)}^2}{A(Z)}
\big( 
 1-\frac{{(1\mp c)}^2 A(Z)}{{(U_s-c)}^2\gamma_\pm^2}
\big)- {(U_s-c)}^2\big({A(Z)}^{-1}-\gamma_\pm^{-2} \big)\big]dZ\\
 =&\frac{{(U_s-c)}^2}{A(Y)}
  \int_{\pm \infty}^Y\big(
 1-\frac{{(1\mp c)}^2 A(Z)}{{(U_s-c)}^2\gamma_\pm^2}\big)dZ
+ \int_{\pm \infty}^Y \frac{{(U_s-c)}^2-{(1\mp c)}^2}{\gamma_\pm^2}
dZ.
\end{align*}
Thus,  we have 
\begin{equation}
\begin{aligned}
G_\pm(Y)
=\int_{\pm \infty}^Y\big(
 1-\frac{{(1\mp c)}^2 A(Z)}{{(U_s-c)}^2\gamma_\pm^2}\big)dZ
+  \frac{A(Y)}{{(U_s-c)}^2} \int_{\pm \infty}^Y \frac{{(U_s-c)}^2-{(1\mp c)}^2}{\gamma_\pm^2}
dZ.
\label{CR46}
\end{aligned}
\end{equation}
Take $\pa_Y$ on \eqref{CR47} to get
\begin{align}\label{pa_Y chi_2}
\partial_Y \chi_{\pm,2}=\partial_Y U_s\int_{\pm \infty}^Y G_\pm(Z)dZ
+(U_s-c)G_\pm(Y).
\end{align}
Take $Y=0$ on both sides of \eqref{CR47} and \eqref{pa_Y chi_2} to have
\begin{equation}
\begin{aligned}
\chi_{\pm,2}(0)=-c
E_\pm,\quad\partial_Y \chi_{\pm,2}(0)=\partial_Y U_s(0)E_\pm
-cG_\pm(0),
\label{CR48}
\end{aligned}
\end{equation}
where $E_\pm= \int_{\pm \infty}^0 G_\pm(Z)dZ.$

$\bullet$ For $\chi_{\pm,n}$ with $n\geq 3$,  using Proposition \ref{pro: varphi} to get
$$\chi_{\pm,n}(0)=O(\lambda^n), \ \partial_Y \chi_{\pm,n}(0)=O(\lambda^n), \ \text{for}\  n \geq 3.$$

$\bullet$
Bringing the above estimates together and using $|\beta_\pm|\lesssim \al$, we get
\begin{align}  
\chi_{\pm}(0)=&-c\pm \beta_\pm (-cB_\pm)-c
\beta_\pm^2 E_\pm +O(\alpha^3  \lambda^3)   \label{RR1} \\
=&
-c\big(
      1\pm \beta_\pm B_\pm+
\beta_\pm^2 E_\pm
    \big)+O(\alpha^3  \lambda^3), \nonumber
    \end{align}
    and
    \begin{align}
 \partial_Y \chi_{\pm}(0)=&\partial_YU_s(0)
\pm \beta_\pm \big(\partial_Y U_s(0)B_\pm-cD_\pm\big)
+\beta_\pm^2
\big( 
\partial_Y U_s(0)E_\pm
-cG_\pm(0)
\big)+O(\alpha^3  \lambda^3)\nonumber\\
=& \partial_YU_s(0)\big(
      1\pm \beta_\pm B_\pm+
\beta_\pm^2 E_\pm
    \big)
    -c
    \big(\pm \beta_\pm D_\pm+\beta_\pm^2  G_\pm(0) 
    \big)+O(\alpha^3  \lambda^3).\label{RR1-1}
\end{align}
On the one hand, we notice that
\[
\partial_Y \varphi_{\pm}(0)=\mp\beta_\pm \chi_{\pm}(0)+\partial_Y \chi_{\pm}(0),
\]
which implies that $W(\alpha,c)$ is written as 
\begin{equation}\label{eq: W}
\begin{aligned}
W(\alpha,c)&=\begin{vmatrix}
    \chi_+(0) &  \chi_-(0)\\
     \partial_Y \chi_+(0)- \beta_+ \chi_+(0) &  \partial_Y \chi_-(0)+\beta_- \chi-(0)\\
\end{vmatrix}\\
&=
\begin{vmatrix}
    \chi_+(0) &  \chi_-(0)\\
     \partial_Y \chi_+(0) &  \partial_Y \chi_-(0)\\
\end{vmatrix}
+(\beta_++\beta_-) \chi_+(0)  \chi_-(0).
\end{aligned}
\end{equation}
According to \eqref{RR1} and \eqref{RR1-1}, we have
\begin{align*}
\begin{vmatrix}
    \chi_+(0) &  \chi_-(0)\\
     \partial_Y \chi_+(0) &  \partial_Y \chi_-(0)\\
\end{vmatrix}=&
\begin{vmatrix}
  -c & -c\\
    \partial_YU_s(0) &  \partial_YU_s(0)\\
\end{vmatrix}
+\alpha\begin{vmatrix}
   -c\gamma_+B_+ &  c\gamma_-B_- \\
        \partial_YU_s(0) &  \partial_YU_s(0)\\
\end{vmatrix}\\
&+\al\begin{vmatrix}
   -c & -c\\
     \gamma_+ \partial_Y U_s(0)B_+ & -\gamma_- \partial_Y U_s(0)B_-\\
\end{vmatrix}
+\al \begin{vmatrix}
   -c & -c\\
     -\gamma_+ cD_+ &       \gamma_- cD_-\\
\end{vmatrix}\\
&+\alpha^2
\begin{vmatrix}
   -c &  -c \\
        \gamma_+^2
\partial_Y U_s(0)E_+
&  \gamma_-^2
\partial_Y U_s(0)E_-
\\
\end{vmatrix}
+\alpha^2\begin{vmatrix}
   -c &  -c \\
        -\gamma_+^2
cG_+(0)
&  -\gamma_-^2cG_-(0)
\\
\end{vmatrix}\\
&+\al^2\begin{vmatrix}
   -c\gamma_+B_+ &  c\gamma_-B_- \\
     \gamma_+\partial_Y U_s(0)B_+&       -\gamma_- \partial_Y U_s(0)B_-\\
\end{vmatrix}\\
&+\al^2\begin{vmatrix}
   -c\gamma_+B_+ &  c\gamma_-B_- \\
     -\gamma_+cD_+&       \gamma_- cD_-\\
\end{vmatrix}
+\al^2\begin{vmatrix}
   -c\gamma_+^2E_+ &  -c\gamma_-^2E_- \\
        \partial_YU_s(0) &  \partial_YU_s(0)\\
\end{vmatrix}
+O(\alpha^3  \lambda^3)\\
=&\alpha c^2\begin{vmatrix}
   1 & 1\\
     \gamma_+ D_+ &       -\gamma_- D_-\\
\end{vmatrix}
+\alpha^2c^2\big(
\begin{vmatrix}
  1 & 1 \\
        \gamma_+^2
G_+(0)
&   \gamma_-^2
G_-(0)\\
\end{vmatrix}-\gamma_+\gamma_-\begin{vmatrix}
   B_+ &  B_- \\
  D_+&   D_-\\
\end{vmatrix}
\big)\\
&+O(\alpha^3  \lambda^3).
\end{align*}

Bringing the above equality and  \eqref{RR1} into \eqref{eq: W}, we get
\begin{align*}
	W(\alpha,c)
 =&\alpha c^2\begin{vmatrix}
   1 & 1\\
     \gamma_+ D_+ &       -\gamma_- D_-\\
\end{vmatrix}
+\alpha^2c^2\big(
\begin{vmatrix}
  1 & 1 \\
        \gamma_+^2
G_+(0)
&   \gamma_-^2
G_-(0)\\
\end{vmatrix}-\gamma_+\gamma_-\begin{vmatrix}
   B_+ &  B_- \\
  D_+&   D_-\\
\end{vmatrix}
\big)\\
&+c^2 \alpha
(\gamma_+ +\gamma_- )
\big(
        1+ \gamma_+  B_+ \alpha +
\gamma_+^2 E_+ \alpha^2 
    \big)
    \big(
      1- \gamma_-  B_- \alpha +
\gamma_-^2 E_-  \alpha^2
    \big)+O(\alpha^3  \lambda^3),
\end{align*}
which yields that
\begin{equation}
\begin{aligned}
\label{CRR1}
	{(c^2 \alpha)}^{-1} W(\alpha,c)
     \coloneqq &-(\gamma_+ D_+    +  \gamma_-  D_- )
+
(\gamma_+ +\gamma_- )
+K\alpha+O(\alpha^2  c^{-2}\lambda^3)\\
     =&\gamma_+(1-D_+)+\gamma_-(1-D_-)+K\alpha+O(\alpha^2  c^{-2}\lambda^3),
\end{aligned}
\end{equation}
where $K$ is defined by
\begin{align*}
K=	- \gamma_+ \gamma_-\begin{vmatrix}
  B_+   &      B_-\\  
    D_+    
    &  
      D_-  
   \\
\end{vmatrix}
+\begin{vmatrix}
   1 &    1\\  
      \gamma_+^2  G_+(0) 
    &  
   \gamma_-^2   G_-(0) 
   \\
\end{vmatrix}+
(\gamma_+ +\gamma_- )
\big(
         \gamma_+  B_+ - \gamma_-  B_-
    \big).
\end{align*}

Next, we provide a more precise formulation of the right-hand side of \eqref{CRR1}. 

$\bullet$ For the $O(1)$ term, we have
\begin{align*}
	&\gamma_+(1-D_+)+\gamma_-(1-D_-)=
 \frac{{(1- c)}^2 (1-m^2 c^2)}{c^2\gamma_+}+
 \frac{{(1+ c)}^2 (1-m^2 c^2)}{c^2\gamma_-}\\
 &\quad= \frac{(1-m^2 c^2)}{c^2\gamma_+\gamma_-}
 \big({(1- c)}^2 \gamma_-+
 {(1+ c)}^2\gamma_+
\big)\\
 &\quad=\frac{(1-m^2 c^2)}{c^2\gamma_+\gamma_-\big({(1- c)}^2 \gamma_--
 {(1+ c)}^2\gamma_+
\big)}
\big({(1- c)}^4 \gamma_-^2-
 {(1+ c)}^4\gamma_+^2
\big)\\
&\quad =\frac{4(1-m^2 c^2)
  \big( m^2 {(1-c^2)}^2-2(c^2+1)\big)}{c\gamma_+\gamma_-  \big({(1- c)}^2 \gamma_--
 {(1+ c)}^2\gamma_+
  \big)}.
\end{align*}

$\bullet$ For the $O(\alpha)$ terms, we denote
\begin{align*}
G_\pm(0)=B_\pm+F_\pm,\ F_\pm=  \frac{1-m^2c^2}{c^2} \int_{\pm \infty}^0 \frac{{(U_s-c)}^2-{(1\mp c)}^2}{\gamma_\pm^2}
dZ,
\end{align*}
which means 
\begin{align*}
K=&	 \gamma_+ \gamma_- (1- D_-) B_+
-   \gamma_+^2 F_+-\gamma_+ \gamma_- (1-D_+) B_-
+ \gamma_-^2  F_-.
\end{align*}
By noticing the facts
\begin{align*}
	& \gamma_+ \gamma_- (1- D_-) B_+
-   \gamma_+^2 F_+\\
=&\frac{1-m^2c^2}{c^2} \int_{+ \infty}^0 \frac{  \big({(U_s-c)}^2-{(1- c)}^2  \big)
  \big( \frac{{(1+c)}^2}{\gamma_+ \gamma_-}
-{(U_s-c)}^2
  \big)
}{{(U_s-c)}^2}
dZ,
\end{align*}
and
\begin{align*}
	&\gamma_+ \gamma_- (1-D_+) B_-
- \gamma_-^2  F_-\\
=&\frac{1-m^2c^2}{c^2} \int_{- \infty}^0 \frac{  \big({(U_s-c)}^2-{(1+ c)}^2  \big)
  \big( \frac{{(1-c)}^2}{\gamma_+ \gamma_-}
-{(U_s-c)}^2
\big)
}{{(U_s-c)}^2}
dZ,
\end{align*}
we get
\begin{align*}
K=	& \frac{1-m^2c^2}{c^2} \int_0^{+ \infty} \frac{\big({(U_s-c)}^2-{(1- c)}^2\big)
\big( 
{(U_s-c)}^2
-\frac{{(1+c)}^2}{\gamma_+ \gamma_-}
\big)
}{{(U_s-c)}^2}
dZ\\
&+\frac{1-m^2c^2}{c^2} \int_{- \infty}^0 \frac{\big({(U_s-c)}^2-{(1+ c)}^2\big)
\big( 
{(U_s-c)}^2-
\frac{{(1-c)}^2}{\gamma_+ \gamma_-}
\big)
}{{(U_s-c)}^2}
dZ.
\end{align*}
Multiplying $c^2\al$ on both sides of the above equation \eqref{CRR1}, we get the desired results.

\qed

\section{The proof of Proposition \ref{pro: varphi}}

In this section, we prove Proposition \ref{pro: varphi} for the two cases $0<m<\sqrt 2$ and $m\geq\sqrt 2$ respectively. Define the following functional spaces for the Rayleigh equation:
\begin{equation}\label{Y_th}
\begin{aligned}
&\|{\varphi}\|_{Y_{ \theta,\pm }}\coloneqq 
\|(U_s-c)A(Y){\partial_Y^2 \varphi}\|_{L_{ \theta ,\pm}^{\infty}}+
\|{\partial_Y \varphi}\|_{L_{ \theta ,\pm}^{\infty}}
+
 \|{\varphi}\|_{L_{ \theta,\pm}^{\infty}},
\end{aligned}
\end{equation}
with
\begin{align*}
	&\|f_+\|_{L^\infty_{\theta,+}}\coloneqq \|e^{ \theta Y}f_+\|_{L^\infty(0,+\infty)},
	\quad\|f_{-}\|_{L^\infty_{\theta,-}}\coloneqq \|e^{- \theta Y}f_-\|_{L^\infty(-\infty,0)},
\end{align*}
where $\theta>0$ is the universal constant defined in \eqref{eq:S-A}. It is easy to see  that $\|f\|_{L^\infty(\mathbb{R}_\pm)}\leq \|f\|_{L^\infty_{\theta,\pm}}$.
In this section, we always assume that $(\alpha, c)\in \mathbb H$.

 Firstly, we consider the following non-homogeneous equation
\begin{align}\label{equ: non-homo}
{Ray}_0[\varphi_\pm]=F_\pm,\quad Y\in \mathbb{R}_{\pm},
\end{align}
where $F_{\pm}\in {L^\infty_{\theta,\pm}}$.

Define
\begin{align}\label{CR100}
\varphi_\pm(Y)=(U_s-c)\int_{\pm\infty}^Y\big(
\frac{A(Z)}{{(U_s-c)}^2}
\int_{\pm \infty}^Z F_{\pm}(Z')dZ'\big)
dZ,\ Y\in \mathbb{R}_\pm.
\end{align}
which is a solution to \eqref{equ: non-homo}.

\medskip

Due to Young's inequality
\[\big\|e^{\pm \th Y}
    \int_{\pm \infty}^Y F_{\pm}(Z)dZ\big\|_{L^\infty}\lesssim 
    \big\|e^{\pm \th Y}
    \int_{\pm \infty}^Y \|F_{\pm}\|_{L^\infty_{\theta,\pm}}e^{\mp \theta Z}dZ\big\|_{L^\infty}\lesssim \|F_{\pm}\|_{L^\infty_{\theta,\pm}}.
    \]
We have the following lemma.
\begin{lemma}
    If $F_{\pm}\in {L^\infty_{\theta,\pm}}$, we have  $$\big\|
    \int_{\pm \infty}^Y F_{\pm}(Z)dZ\big\|_{L^\infty_{\theta,\pm}}\lesssim \|F_{\pm}\|_{L^\infty_{\theta,\pm}}.$$
    \label{con:CR2}
\end{lemma}

\medskip

\subsection{The case $m<\sqrt{2}$}

\begin{lemma} \label{con:CR1}
Let $(\alpha,c)\in \mathbb{H}_<$ and $\varphi_{\pm}$ be given in \eqref{CR100}. For each fixed $0<m<\sqrt 2$, it holds that
 \begin{align} 
 \|{\varphi_\pm}\|_{Y_{ \theta,\pm }}\approx\|\partial_Y^2\varphi_\pm\|_{L^\infty_{\theta,\pm}} +\|\partial_Y\varphi_\pm\|_{L^\infty_{\theta,\pm}}+\|\varphi_\pm\|_{L^\infty_{\theta,\pm}}\lesssim \|F_\pm\|_{L^\infty_{\theta,\pm}}.
\end{align} 
  
\end{lemma}
\begin{proof}

According to \eqref{est1: U_s-c, A}, we have
$$|\varphi_\pm(Y)|\lesssim
\big|\int_{\pm\infty}^Y\big(
\int_{\pm \infty}^Z |F_\pm(Z')|dZ'\big)
dZ\big|,$$
which combines with Lemma \ref{con:CR2} to get
\begin{align}\label{est1: varphi}
\|\varphi_\pm\|_{L^\infty_{\theta,\pm}}\lesssim \|F_\pm\|_{L^\infty_{\theta,\pm}}.
\end{align}
Taking one derivative of \eqref{CR100}, we obtain
\begin{align*}
	\partial_Y\varphi_\pm(Y)=\partial_Y U_s\int_{\pm\infty}^Y\big(
\frac{A(Z)}{{(U_s-c)}^2}
\int_{\pm \infty}^Z F_\pm(Z')dZ'\big)
dZ
+
\frac{A(Y)}{{(U_s-c)}}
\int_{\pm \infty}^Y F_\pm(Z)dZ.
\end{align*}
Using \eqref{est1: U_s-c, A} and Lemma \ref{con:CR2}, we have
\begin{align}\label{est1: pa_Y varphi}
\|\partial_Y \varphi_\pm\|_{L^\infty_{\theta,\pm}}\lesssim \|F_\pm\|_{L^\infty_{\theta,\pm}}.
\end{align}

Using the equation \eqref{equ: non-homo} and estimates \eqref{est1: varphi}, \eqref{est1: pa_Y varphi}, \eqref{est1: U_s-c, A} to derive
\[\|\partial_Y^2 \varphi_\pm\|_{L^\infty_{\theta,\pm}}\lesssim \|F_\pm\|_{L^\infty_{\theta,\pm}}.
\]

Combining all the above estimates, we get the desired results.
 
\end{proof}

\medskip

Based on the above lemma, we are in a position to give the proof of Proposition \ref{pro: varphi} for the case $m<\sqrt{2}$.

\underline{{\it Proof of Proposition \ref{pro: varphi} for the case $m<\sqrt{2}$.}}

 By \eqref{CR4} and  $\chi_{\pm,0}=U_s-c$, we have  
  $$\|{Ray}_0[\chi_{\pm,1}]\|_{L^\infty_{\theta,\pm}}\lesssim 1.$$
  which, combined with Lemma \ref{con:CR1}, gives
   $$\|\chi_{\pm,1}\|_{Y_{\theta,\pm}} \lesssim 1.$$

Noticing
\[\big\|(U_s-c)\big({A(Y)}^{-1}-\gamma_\pm^{-2} \big)\chi_{\pm,0}\big\|_{L^\infty_{\theta,\pm}}\lesssim 1,
\]
we use \eqref{def: F_n+2} to get
\[
\|{Ray}_0[\chi_{\pm,2}]\|_{L^\infty_{\theta,\pm}}\lesssim 1,
\]
Applying Lemma \ref{con:CR1} to have 
\beno
\|\chi_{\pm,2}\|_{Y_{\theta,\pm}} \lesssim 1.
\eeno

For $\chi_{\pm,n}$ with $n\geq 3$, we use the iterative method to get their estimates. Suppose
 \beno
 \chi_{\pm,n},\quad \chi_{\pm,n+1},\quad \partial_Y\chi_{\pm,n+1}\in {L^\infty_{\theta,\pm}},
 \eeno
 by \eqref{CR4} and \eqref{def: F_n+2}, we get
\beno
\|{Ray}_0[\chi_{\pm,n+2}]\|_{L^\infty_{\theta,\pm}} \lesssim\|\partial_Y \chi_{\pm,n+1}\|_{L^\infty_{\theta,\pm}}+\|\chi_{\pm,n}\|_{L^\infty_{\theta,\pm}} 
+\|\chi_{\pm,n+1}\|_{L^\infty_{\theta,\pm}} .
\eeno
Suppose
\beno
\chi_{\pm,n},\quad \chi_{\pm,n+1}\in {Y_{\theta,\pm}},
\eeno 
by Lemma \ref{con:CR1}, we know
\beno
\|\chi_{\pm,n+2}\|_{Y_{\theta,\pm}} \lesssim \|\chi_{\pm,n}\|_{Y_{\theta,\pm}}
+\|\chi_{\pm,n+1}\|_{Y_{\theta,\pm}}.
\eeno
Combining all the above estimates, we get
\begin{align*}
\|\chi_{\pm,1}\|_{Y_{\theta,\pm}} \lesssim 1,\quad
    \|\chi_{\pm,2}\|_{Y_{\theta,\pm}} \lesssim 1,
\end{align*}
and
\begin{align*}    
     \|\chi_{\pm,n+2}\|_{Y_{\theta,\pm}} \lesssim\|\chi_{\pm,n}\|_{Y_{\theta,\pm}}
+\|\chi_{\pm,n+1}\|_{Y_{\theta,\pm}}
, \ n\geq 1,
\end{align*}
which implies that
\beno
\|\chi_{\pm,n}\|_{Y_{\theta,\pm}}\leq C^n , \ n \geq 1.
\eeno

Recalling that
\begin{equation}
\begin{aligned}
\chi_{\pm}(Y)=\sum_{n=0}^{\infty} {(\pm  \beta_\pm)}^n \chi_{\pm,n}(Y),\quad \chi_\pm (Y)-(U_s(Y)-c)=\sum_{n=1}^{\infty} {(\pm  \beta_\pm)}^n \chi_{\pm,n}(Y),
\end{aligned}
\end{equation}
we get
\begin{align*}
\|\chi_\pm (Y)-(U_s(Y)-c)\|_{Y_{\th,\pm}}\lesssim \sum_{n=1}^{\infty}\al^n \lesssim \al,
\end{align*}
when $\al$ is small enough. Thus, we get that $\chi_{\pm}$ is well-defined in $Y_{\th,\pm}$, which completes the proof.
 
 \qed

\medskip

\subsection{The case $m\geq \sqrt{2}$}
We first introduce two points $Y_c$ and $\widetilde{Y_{c,\pm}}$ which satisfy that
\beno
U_s(Y_c)=c_r,\quad U_s(\widetilde{Y_{c,\pm}})=  \pm \frac{1}{m} +c_r.
\eeno
Moreover, $\widetilde{Y_{c,\pm}}$ is the root of the main part of $\Re A(Y)$ (i.e., $1-m^2(U_s(\widetilde{Y_{c,\pm}})-c_r)^2=0$).

When $m\geq \sqrt{2}$, we know $|U_s-c|$ and $|A(Y)|$ have a lower bound $c_i$ instead of $1$ (see \eqref{est2: U_s-c, A}).  By the same argument as in Lemma \ref{con:CR2}, we have 
\begin{lemma}  \label{con:CR4}
Let $(\alpha,c)\in \mathbb{H}_\geq$,  $F_{\pm}\in {L^\infty_{\theta,\pm}}$. For each fixed $m\geq \sqrt2$, it holds 
\begin{align}
\big\|
    \int_{\pm \infty}^Y \frac{F_{\pm}(Z)}{U_s(Z)-c}dZ\big\|_{L^\infty_{\theta,\pm}}&\lesssim |\log c_i|\|F_{\pm}\|_{L^\infty_{\theta,\pm}},\label{ineq1}\\
\big\|
    \int_{\pm \infty}^Y \frac{F_{\pm}(Z)}{A(Z)}dZ\big\|_{L^\infty_{\theta,\pm}}&\lesssim |\log c_i|\|F_{\pm}\|_{L^\infty_{\theta,\pm}},\label{ineq2}
\end{align}
 and
\begin{align}\label{ineq3}
\big\|
    \int_{\pm \infty}^Y \frac{F_{\pm}(Z)}{A(Z)(U_s(Z)-c)}dZ\big\|_{L^\infty_{\theta,\pm}}\lesssim |\log c_i|\|F_{\pm}\|_{L^\infty_{\theta,\pm}}.
    \end{align}
 
\end{lemma}

\begin{proof}
A direct calculation gives that
\begin{align*}
\f{1}{U_s(Y)-c}=&\f{1}{U_s(Y)-U_s(Y_c)-ic_i},\\
\f{1}{A(Y)}=&\f{1}{2m}\left(\f{1}{U_s(Y)-c+\f{1}{m}}-\f{1}{U_s(Y)-c-\f{1}{m}}\right),\\
=&\f{1}{2m}\left(\f{1}{U_s(Y)-U_s(\widetilde{Y_{c,-}})-ic_i}-\f{1}{U_s(Y)-U_s(\widetilde{Y_{c,+}})-ic_i}\right),\\
\f{1}{A(Y)(U_s(Y)-c)}=&\f{1}{U_s(Y)-c}-\f12\f{1}{U_s(Y)-c-\f1m}-\f12\f{1}{U_s(Y)-c+\f1m}\\
=&\f{1}{U_s(Y)-U_s(Y_c)-ic_i}-\f12\f{1}{U_s(Y)-U_s(\widetilde{Y_{c,+}})-ic_i}-\f12\f{1}{U_s(Y)-U_s(\widetilde{Y_{c,-}})-ic_i}
\end{align*}
For $(\alpha,c)\in \mathbb{H}_\geq$ and fixed $m\geq \sqrt2$, we have 
\begin{align*}
&-\f12<\widetilde{Y_{c,-}}<Y_c<\widetilde{Y_{c,+}}<1,\\
 &U_s'(Y_{c,-}),~U_s'(Y_{c}),~U_s'(Y_{c,+})\geq c_0>0.
\end{align*}
For this reason, we only prove the case 
\begin{align*}
\| \int_{+ \infty}^Y \frac{F_{+}(Z)}{U_s(Z)-c}dZ\big\|_{L^\infty_{\theta,+}}&\lesssim |\log c_i|\|F_{+}\|_{L^\infty_{\theta,+}},
\end{align*}
in \eqref{ineq1}. The estimates in \eqref{ineq2}-\eqref{ineq3} can be handled by the same process. 

For $Y \geq Y_c+1$, we notice $|U_s-c|\approx 1$. By Lemma \ref{con:CR2}, we get
$$\big|
    \int_{+ \infty}^Y \frac{F_{+}(Z)}{U_s(Z)-c}dZ\big|
   \lesssim \int_{+ \infty}^Y |F_{+}(Z)|dZ
    \lesssim \|F_{+}\|_{L^\infty_{\theta,+}}e^{-\theta Y}.$$
For $0 \leq Y\leq Y_{c}+1$, noticing $e^{\mp \theta Y}\approx 1$, we have
\begin{equation}
\begin{aligned}
\big|
    \int_{+ \infty}^Y \frac{F_{+}(Z)}{U_s(Z)-c}dZ\big|
  & \lesssim 
    \int_{Y_c+1}^{+ \infty}\big|\frac{F_{+}(Z)}{U_s(Z)-c}\big|dZ
   +
    \int_{0}^{Y_c+1} \big|\frac{F_{+}(Z)}{U_s(Z)-c}\big|dZ\\
   & \lesssim \|F_{+}\|_{L^\infty_{\theta,+}}e^{-\theta Y_{0,+}}+
    \|F_{+}\|_{L^\infty}\int_{0}^{Y_c+1} \big|\frac{1}{U_s(Z)-c}\big|dZ.
    \label{CR21}
\end{aligned}
\end{equation}
For $0\leq Z \leq Y_c+1$, by the Mean Value Theorem, we notice  $|U_s(Z)-c_r|\approx |Z-Y_c|$, which implies $|U_s(Z)-c|\approx |Z-Y_c|+c_i$. Therefore, we obtain
\begin{equation}
\begin{aligned}
\int_{0}^{Y_c+1} \big|\frac{1}{U_s(Z)-c}\big|dZ
\lesssim \int_{0}^{2Y_c+1} \frac{1}{Z+c_i}dZ
\lesssim |\log c_i|.
\label{CR22}
\end{aligned}
\end{equation}
By \eqref{CR21} and \eqref{CR22}, we derive
\begin{equation*}
\begin{aligned}
\big|
    \int_{+ \infty}^Y \frac{F_{+}(Z)}{U_s(Z)-c}dZ\big|
   \lesssim |\log c_i| \|F_{+}\|_{L^\infty_{\theta,+}} , \ Y\leq Y_c+1.
\end{aligned}
\end{equation*}

Combining all the above estimates, we get
$$\big|
    \int_{+ \infty}^Y \frac{F_{+}(Z)}{U_s(Z)-c}dZ\big|
    \lesssim |\log c_i|\|F_{+}\|_{L^\infty_{\theta,+}}e^{-\theta Y}, \ Y \in \mathbb{R}_+.$$
Here, we finish the proof.
\end{proof}
\begin{lemma}   \label{con:G}
Let $(\alpha,c)\in \mathbb{H}_\geq$ and $G(Y)\in {C^1[Y, Y_1]}$ with $Y<Y_c<Y_1$. It holds that 
\begin{align*}
\big|\int^{Y_1}_Y\frac{G(Z)}{{(U_s-c)}^2}dZ-\frac{G(Y)}{(U_s-c)\partial_YU_s}\big|\lesssim|\log c_i|(\|G\|_{L^\infty([Y, Y_1])}+\|\pa_YG\|_{L^\infty([Y, Y_1])}).
\end{align*} 
\end{lemma}

\begin{proof}
By integration by parts, we have
\begin{align*}
\int^{Y_1}_Y\frac{G(Z)}{{(U_s-c)}^2}dZ=-\frac{G(Y)}{(U_s-c)\partial_YU_s}\Bigg|^{Y_1}_Y+\int^{Y_1}_Y\frac{1}{{(U_s-c)}}\partial_Z\big(\frac{G}{\partial_YU_s}\big)dZ,
\end{align*}
which implies that
\begin{align*}
&\big|\int^{Y_1}_Y\frac{G(Z)}{{(U_s-c)}^2}dZ-\frac{G(Y)}{(U_s-c)\partial_YU_s}\big|\\
\lesssim& \big|\f{G(Y_1)}{(U_s(Y_1)-c)\pa_Y U_s(Y_1)}\big| +(\|G\|_{L^\infty([Y, Y_1])}+\|\pa_YG\|_{L^\infty([Y, Y_1])})\int_{Y}^{Y_1}\f{1}{|U_s-c|}dZ\\
\lesssim& |\log c_i|(\|G\|_{L^\infty([Y, Y_1])}+\|\pa_YG\|_{L^\infty([Y, Y_1])}).
\end{align*}
 
\end{proof}

\medskip

According to Lemma \ref{con:CR4}, we derive the following two lemmas.

\begin{lemma} \label{con:CR7}
Let $\chi_{\pm,n},\chi_{\pm,n+1},\partial_Y  \chi_{\pm,n+1}\in {L^\infty_{\theta,\pm}} $.  Then for each fixed $m\geq \sqrt2$, we have
\beno
\|\chi_{\pm,n+2}\|_{Y_{\theta,\pm}} \lesssim |\log c_i|^2
\big(
\|\partial_Y \chi_{\pm,n+1}\|_{L^\infty_{\theta,\pm}}+\|\chi_{\pm,n+1}\|_{L^\infty_{\theta,\pm}} 
+\|\chi_{\pm,n}\|_{L^\infty_{\theta,\pm}} \big).
\eeno
Especially,  when
 $\chi_{\pm,n},\chi_{\pm,n+1}\in {Y_{\theta,\pm}} $, we have
\beno
\|\chi_{\pm,n+2}\|_{Y_{\theta,\pm}} \lesssim |\log c_i|^2
\big(
\| \chi_{\pm,n+1}\|_{Y_{\theta,\pm}}
+\|\chi_{\pm,n}\|_{Y_{\theta,\pm}} \big).
\eeno
\end{lemma}

\medskip

\begin{lemma} \label{con:CR6} 
For each fixed $m\geq \sqrt2$,      we have $\chi_{\pm,1},\partial_Y  \chi_{\pm,1},\chi_{\pm,2},\partial_Y  \chi_{\pm,2}\in {L^\infty_{\theta,\pm}} $. Moreover, we get
\begin{equation}
\begin{aligned}
&\| \chi_{\pm,1}\|_{Y_{\theta,\pm}}
\lesssim
{|\log c_i|}, \quad
\| \chi_{\pm,2}\|_{Y_{\theta,\pm}}
\lesssim
{|\log c_i|}^2.
\end{aligned}
\end{equation}
\end{lemma}


We postpone the proof of Lemma  \ref{con:CR7} and Lemma  \ref{con:CR6} to a later section, and present the proof of Proposition \ref{pro: varphi} for the case $m\geq \sqrt{2}$ by using Lemma  \ref{con:CR7} and Lemma  \ref{con:CR6}. 
 
 \bigskip

\underline{\it Proof of Proposition \ref{pro: varphi} for the case $m\geq \sqrt{2}$.}

\begin{proof}
   
By Lemma \ref{con:CR7} and Lemma \ref{con:CR6}, we know the sequence $\{\|\chi_{\pm,n}\|_{Y_{\theta,\pm}} 
\}$ satisfies
\begin{equation*}
\begin{aligned}
\begin{cases}
&\|\chi_{\pm,n+2}\|_{Y_{\theta,\pm}} \lesssim {|\log c_i|^2}
\big(\|\chi_{\pm,n}\|_{Y_{\theta,\pm}}
+\|\chi_{\pm,n+1}\|_{Y_{\theta,\pm}}\big)
, \ n\geq 1, \\
& \|\chi_{\pm,1}\|_{Y_{\theta,\pm}} \lesssim {|\log c_i|},\\
    &\|\chi_{\pm,2}\|_{Y_{\theta,\pm}} \lesssim {|\log c_i|}^2,
\end{cases}
\end{aligned}
\end{equation*}
which implies that
\begin{equation*}
\begin{aligned}
\|\chi_{\pm,n}\|_{Y_{\theta,\pm}}\leq {\big(C|\log c_i|\big)}^{2n}, \ n \geq 1.
\end{aligned}
\end{equation*}
Recall that
\begin{equation*}
\begin{aligned}
\chi_{\pm}(Y)=\sum_{n=0}^{\infty} {(\pm  \beta_\pm)}^n \chi_{\pm,n}(Y), \quad\chi_\pm (Y)-(U_s(Y)-c)=\sum_{n=1}^{\infty} {(\pm  \beta_\pm)}^n \chi_{\pm,n}(Y),
\end{aligned}
\end{equation*}
we get
\begin{align*}
\|\chi_\pm (Y)-(U_s(Y)-c)\|_{Y_{\th,\pm}}\lesssim \sum_{n=1}^{\infty}\al^n|\log c_i|^{2n}\lesssim \al |\log c_i|^2,
\end{align*}
where we used that $\al|\log c_i|^2$ is small for $(\al,c)\in \mathbb{H}_\geq$.

\end{proof}

\subsection{Proof of  Lemma  \ref{con:CR7} }

In this subsection, we present the proof of Lemma  \ref{con:CR7}.

For convenience, define
\begin{align*}
 R(Y)\coloneqq&\big(\|\chi_{\pm,n+1}\|_{L^\infty_{\theta,\pm}}+\|\partial_Y\chi_{\pm,n+1}\|_{L^\infty_{\theta,\pm}}+\|\chi_{\pm,n}\|_{L^\infty_{\theta,\pm}} \big)
e^{\mp\theta Y},\\
 H_{n+1}(Y)\coloneqq& \int_{\pm \infty}^Y
(U_s-c)\big( 2{A(Z)}^{-1}\partial_Z\chi_{\pm,n+1}+\partial_Z({A(Z)}^{-1})\chi_{\pm,n+1}\big)dZ,\\
P_{n}(Y)\coloneqq &\int_{\pm \infty}^Y
(U_s-c)\big({A(Z)}^{-1}-\gamma_\pm^{-2} \big)\chi_{\pm,n}dZ\\
=&\int_{\pm \infty}^Y
(U_s-c)
\frac{m^2 
\big({(U_s-c)}^2-{(1\mp c)}^2 \big)
}{\gamma_\pm^{2} A(Z)}
\chi_{\pm,n}dZ.
\end{align*}

Next, we establish the asymptotic properties of  $H_{n+1}$. Integrating by parts to get
\begin{equation*}
\begin{aligned}
H_{n+1}(Y)=
&\int_{\pm \infty}^Y
2(U_s-c) {A(Z)}^{-1}\partial_Z\chi_{\pm,n+1}dZ
+\int_{\pm \infty}^Y (U_s-c)\chi_{\pm,n+1}d({A(Z)}^{-1})\\
=& \frac{(U_s-c)\chi_{\pm,n+1}}{A(Y)}+
\int_{\pm \infty}^Y
\frac{
(U_s-c) \partial_Z\chi_{\pm,n+1}
-\partial_Z U_s \chi_{\pm,n+1}}{A(Z)}dZ.
\end{aligned}
\end{equation*}
Then,  applying  Lemma \ref{con:CR4} and $|U_s-c|\lesssim 1$ to have
\begin{equation}
\begin{aligned}
&\big|H_{n+1}(Y)-\frac{(U_s(Y)-c)\chi_{\pm,n+1}(Y)}{A(Y)}\big| \lesssim|\log c_i|
\big(\|\chi_{\pm,n+1}\|_{L^\infty_{\theta,\pm}} 
+\|\partial_Y\chi_{\pm,n+1}\|_{L^\infty_{\theta,\pm}}\big)
e^{\mp\theta Y},\\
&\big|\pa_Y\big(H_{n+1}(Y)-\frac{(U_s(Y)-c)\chi_{\pm,n+1}(Y)}{A(Y)}\big)\big| \lesssim|A(Y)|^{-1}
\big(\|\chi_{\pm,n+1}\|_{L^\infty_{\theta,\pm}} 
+\|\partial_Y\chi_{\pm,n+1}\|_{L^\infty_{\theta,\pm}}\big)
e^{\mp\theta Y}.
\label{CR24}
\end{aligned}
\end{equation}

In addition, by Lemma \ref{con:CR4} and $|U_s-c|\lesssim 1$, we have
\begin{equation}
\begin{aligned}
|P_{n}(Y)|
\lesssim |\log c_i|
\|\chi_{\pm,n}\|_{L^\infty_{\theta,\pm}} 
e^{-\theta Y},\quad 
|\pa_YP_{n}(Y)|
\lesssim |A(Y)|^{-1}|U_s-c|
\|\chi_{\pm,n}\|_{L^\infty_{\theta,\pm}} 
e^{\mp\theta Y}.
\label{CR25}
\end{aligned}
\end{equation}

Combining \eqref{CR24} with \eqref{CR25} to obtain
 \begin{align}
&\big |H_{n+1}(Y)-P_{n}(Y)-\frac{(U_s(Y)-c)\chi_{\pm,n+1}(Y)}{A(Y)}\big|\lesssim |\log c_i| R(Y),\label{est1: H-P }\\
&\big |\pa_Y\big(H_{n+1}(Y)-P_{n}(Y)-\frac{(U_s(Y)-c)\chi_{\pm,n+1}(Y)}{A(Y)}\big)\big|\lesssim |A(Y)|^{-1} R(Y).\label{est2: H-P }
\end{align}

Rewrite the formula of  $\chi_{\pm ,n+2}$ as follows
\begin{align}
\chi_{\pm ,n+2}=&(U_s-c)\int_{\pm \infty}^Y\frac{A(Z)}{{(U_s-c)}^2}\big( H_{n+1}-P_{n}\big)dZ\label{CR42}\\
=&(U_s-c)\int_{\pm \infty}^Y\frac{\chi_{\pm,n+1}(Z)}{U_s-c}dZ\nonumber\\
&+(U_s-c)
\int_{\pm \infty}^{Y_c\pm 1}
\frac{A(Z)}{{(U_s-c)}^2}
\big( H_{n+1}(Y)-P_{n}(Y)-\frac{(U_s(Y)-c)\chi_{\pm,n+1}(Y)}{A(Y)}\big)dZ\nonumber\\
&+(U_s-c)
\int_{Y_c\pm 1}^Y
\frac{A(Z)}{{(U_s-c)}^2}
\big( H_{n+1}(Y)-P_{n}(Y)-\frac{(U_s(Y)-c)\chi_{\pm,n+1}(Y)}{A(Y)}\big)dZ\nonumber\\
=&I_1+I_2+I_3.\nonumber
\end{align}
Next, we give the estimates $I_i$ one by one. According to Lemma \ref{con:CR4}, we get
\begin{align*}
|I_1|\lesssim |\log c_i| \|\chi_{\pm,n+1}\|_{L^\infty_{\theta,\pm}} e^{\pm \th Y} \lesssim |\log c_i| R(Y).
\end{align*}
For $Z\in[Y_c+ 1,+\infty)\cup(-\infty, Y_c- 1]$, we have \beno
|U_s(Z)-c|\gtrsim1.
\eeno
  Based on the above estimates and \eqref{est1: H-P }, we deduce that
\begin{align*}
|I_2|\lesssim |\log c_i| R(Y).
\end{align*}
 For $I_3$, we take 
 \beno
 G(Y)=A(Y)(H_{n+1}(Y)-P_{n}(Y)-\frac{(U_s(Y)-c)\chi_{+,n+1}(Y)}{A(Y)}),\quad \textrm{for}\quad Z\in[Y, Y_c+ 1]
 \eeno
 or
\beno
G(Y)=A(Y)(H_{n+1}(Y)-P_{n}(Y)-\frac{(U_s(Y)-c)\chi_{-,n+1}(Y)}{A(Y)}),\quad \textrm{for}\quad Z\in[Y_c- 1, Y].
\eeno
By Lemma \ref{con:G}, \eqref{est1: H-P } and \eqref{est2: H-P }, we get 
 \begin{align}\label{est: G}
 \|G\|_{L^\infty[Y, Y_c+ 1]\cup [Y_c- 1, Y]}\lesssim |\log c_i| R(Y),\quad \|\pa_YG\|_{L^\infty[Y, Y_c+ 1]\cup [Y_c- 1, Y]}\lesssim |\log c_i| R(Y),
 \end{align}
which gives that
\begin{align*}
|I_3|\lesssim&|\log c_i|^2 R(Y).
\end{align*}
 Putting the estimates of $I_1$--$I_3$ together, we get
 \begin{align*}
 |\chi_{\pm ,n+2}|\lesssim&|\log c_i|^2 R(Y),
 \end{align*}
which implies that
\begin{align*}
\|\chi_{\pm ,n+2}\|_{L^\infty_{\theta,\pm}}
\lesssim
|\log c_i|^2
\big(\|\chi_{\pm,n+1}\|_{L^\infty_{\theta,\pm}} 
+\|\partial_Y\chi_{\pm,n+1}\|_{L^\infty_{\theta,\pm}}
+\|\chi_{\pm,n}\|_{L^\infty_{\theta,\pm}} \big).
\end{align*}

\medskip

Next, we estimate $\pa_Y \chi_{\pm ,n+2}$. A direct calculation of $I_1$ to have
\begin{align*}
\pa_Y I_1=\chi_{\pm, n+1}+\pa_Y U_s\int_{\pm \infty}^Y\frac{\chi_{\pm,n+1}(Z)}{U_s-c}dZ.
\end{align*}
By Lemma \ref{con:CR4}, we get
\begin{align*}
|\pa_Y I_1|\lesssim |\chi_{\pm, n+1}|+|\log c_i| \|\chi_{\pm,n+1}\|_{L^\infty_{\theta,\pm}} e^{\pm \th Y} \lesssim |\log c_i| R(Y).
\end{align*}
Taking the $\pa_Y$ derivative of $I_2$, we have
\begin{align*}
\pa_Y I_2=\pa_YU_s
\int_{\pm \infty}^{Y_c\pm 1}
\frac{A(Z)}{{(U_s-c)}^2}
\big( H_{n+1}(Y)-P_{n}(Y)-\frac{(U_s(Y)-c)\chi_{\pm,n+1}(Y)}{A(Y)}\big)dZ.
\end{align*}
Applying \eqref{est1: H-P }, we deduce
\begin{align*}
|\pa_Y I_2|\lesssim  |\log c_i| R(Y).
\end{align*}
Taking the derivative $\pa_Y$ of $I_3$, we get
\begin{align*}
\pa_Y I_3=&\pa_YU_s\Big[
\int_{Y_c\pm 1}^Y
\frac{A(Z)}{{(U_s-c)}^2}
\big( H_{n+1}(Y)-P_{n}(Y)-\frac{(U_s(Y)-c)\chi_{\pm,n+1}(Y)}{A(Y)}\big)dZ\\
&+\frac{A(Z)}{{U_s-c}}
\big( H_{n+1}(Y)-P_{n}(Y)-\frac{(U_s(Y)-c)\chi_{\pm,n+1}(Y)}{A(Y)}\big)\Big]=\pa_YU_sI_{3,1}.
\end{align*}
Using estimate \eqref{est: G}, we derive that
\begin{align*}
|\pa_Y I_3|\lesssim |\log c_i|^2 R(Y).
\end{align*}
Putting $\pa_YI_1$--$\pa_Y I_3$ together, we obtain
\begin{align*}
|\pa_Y\chi_{\pm ,n+2}|\lesssim&|\log c_i|^2 R(Y),
\end{align*}
which implies that
\begin{align*}
\|\pa_Y\chi_{\pm ,n+2}\|_{L^\infty_{\theta,\pm}}
\lesssim
|\log c_i|^2
\big(\|\chi_{\pm,n+1}\|_{L^\infty_{\theta,\pm}} 
+\|\partial_Y\chi_{\pm,n+1}\|_{L^\infty_{\theta,\pm}}
+\|\chi_{\pm,n}\|_{L^\infty_{\theta,\pm}} \big).
\end{align*}

\medskip

For $\partial_Y^2 \chi_{\pm ,n+2}$, we use the equation $ {Ray}_0[\chi_{\pm,n+2}]=F_{\pm,n+2}$ with 
\begin{align*}
\|A(Y)^2 F_{\pm,n+2}\|_{L^\infty_{\theta,\pm}}\lesssim \|\chi_{\pm,n+1}\|_{L^\infty_{\theta,\pm}} 
+\|\partial_Y\chi_{\pm,n+1}\|_{L^\infty_{\theta,\pm}}
+\|\chi_{\pm,n}\|_{L^\infty_{\theta,\pm}},
\end{align*}
and 
\begin{align*}
A(Y)^2{Ray}_0[\chi_{\pm,n+2}]=&A(Y)(U_s-c)\partial_Y^2 \chi_{\pm,n+2}-(U_s-c)\partial_Y A(Y)\partial_Y \chi_{\pm,n+2}\\
&-A(Y)\partial_Y^2 U_s\chi_{\pm,n+2}+\partial_Y A(Y)\partial_Y U_s \chi_{\pm,n+2}
\end{align*}
to get
\begin{align*}
\|(U_s-c)A(Y)\partial_Y^2 \chi_{\pm ,n+2}\|_{L^\infty_{\theta,\pm}} \lesssim
|\log c_i|^2
\big(\|\chi_{\pm,n+1}\|_{L^\infty_{\theta,\pm}} 
+\|\partial_Y\chi_{\pm,n+1}\|_{L^\infty_{\theta,\pm}}
+\|\chi_{\pm,n}\|_{L^\infty_{\theta,\pm}} \big).
\end{align*}

By now, we have finished the proof of Lemma \ref{con:CR7}.

\medskip

\subsection{Proof of  Lemma  \ref{con:CR6} }

We first give an estimate for $\chi_{\pm,1}$. Recall
\begin{equation}
\begin{aligned}
\chi_{\pm,1}
 &=(U_s-c) \int_{\pm \infty}^Y 
 \frac{{(U_s-c)}^2-{(1\mp c)}^2}{{(U_s-c)}^2\gamma_\pm^2}
 dZ.
\end{aligned}
\end{equation}
For $Y\in [Y_c+ 1,+\infty)\cup (-\infty,Y_c- 1]$, we know $|U_s(Y)-c| \approx 1, \ |A(Y)| \approx 1$. By Lemma \ref{con:CR2}, we have
\begin{equation}
\begin{aligned}
|\chi_{\pm,1}(Y)|&\lesssim
\big\|{(U_s-c)}^2-{(1\mp c)}^2\big\|_{L^\infty_{\theta,\pm}}
e^{\mp\theta Y}\\
&=\big\|(U_s\mp 1)(U_s-2c\pm 1)\big\|_{L^\infty_{\theta,\pm}}
e^{\mp\theta Y} \\
&\lesssim 
e^{\mp\theta Y}.
\label{CR30}
\end{aligned}
\end{equation}
For $Y\in [Y_c- 1, Y_c+ 1]$, noticing $e^{\mp \theta Y}\approx 1$ and \eqref{CR30}, we get
\beno
|\chi_{\pm,1}(Y_c\pm 1)| \lesssim  1.
\eeno
 We apply Lemma \ref{con:G} with $G=\f{{(U_s-c)}^2-{(1\mp c)}^2}{\gamma_\pm^2}$
to have
\begin{equation}\label{est: chi_1-2}
\begin{aligned}
\chi_{\pm,1}(Y)
 =\chi_{\pm,1}(Y_c\pm 1)-\frac{{(U_s-c)}^2-{(1\mp c)}^2}{\partial_Y U_s}+O(|\log c_i|)=O(|\log c_i|).
\end{aligned}
\end{equation}
By \eqref{CR30} and \eqref{est: chi_1-2}, we derive
\begin{equation}
\begin{aligned}
|\chi_{\pm,1}(Y)|\lesssim
|\log c_i|
e^{\mp\theta Y}, \ Y\in \mathbb{R}_\pm,
\end{aligned}
\end{equation}
which implies
\begin{equation}
\begin{aligned}
\label{C4}
\|\chi_{\pm,1}\|_{L^\infty_{\theta,\pm}}
\lesssim
|\log c_i|
.
\end{aligned}
\end{equation}
 
 Since
\beno
\partial_Y \chi_{\pm,1}
 =\partial_Y U_s \int_{\pm \infty}^Y 
 \frac{{(U_s-c)}^2-{(1\mp c)}^2}{{(U_s-c)}^2\gamma_\pm^2}
 dZ
 +
 \frac{{(U_s-c)}^2-{(1\mp c)}^2}{{(U_s-c)}\gamma_\pm^2}.
\eeno
For $Y\in [Y_c+ 1,+\infty)\cup (-\infty,Y_c- 1]$, we know $|U_s-c| \approx 1$. By Lemma \ref{con:CR2}, we have
\begin{equation}
\begin{aligned}
|\partial_Y \chi_{\pm,1}(Y)|\lesssim
\big\|{(U_s-c)}^2-{(1\mp c)}^2\big\|_{L^\infty_{\theta,\pm}}
e^{\mp\theta Y}
\lesssim 
e^{\mp\theta Y}.
\label{CR31}
\end{aligned}
\end{equation}
For $Y\in [Y_c- 1, Y_c+ 1]$, we notice $e^{\mp \theta Y}\approx 1$. By \eqref{CR31}, we get
\beno
|\partial_Y \chi_{\pm,1}(Y_c\pm 1)|
\lesssim 
1.
\eeno
We apply Lemma \ref{con:G} with $G=\f{{(U_s-c)}^2-{(1\mp c)}^2}{\gamma_\pm^2}$
to have
\begin{equation*}
\begin{aligned}
\partial_Y \chi_{\pm,1}(Y)
 =&\partial_Y \chi_{\pm,1}(Y_c\pm 1)-\frac{{(U_s-c)}^2-{(1\mp c)}^2}{{(U_s-c)}\gamma_\pm^2}
\\&+O(|\log c_i|) +\frac{{(U_s-c)}^2-{(1\mp c)}^2}{{(U_s-c)}\gamma_\pm^2}=O(|\log c_i|)
 .
\end{aligned}
\end{equation*}
Hence, we have
\begin{equation}
\begin{aligned}
\label{C5}
\|\partial_Y \chi_{\pm,1}\|_{L^\infty_{\theta,\pm}}
\lesssim
|\log c_i|
.
\end{aligned}
\end{equation}
By \eqref{CR4}, \eqref{C4}, \eqref{C5} and the equation of $\chi_{\pm,1}:$
\begin{align*}
 {Ray}_0[\chi_{\pm,1}]=(U_s-c)\big( 2A^{-1}\partial_Y\chi_{\pm,0}+\partial_Y(A^{-1})\chi_{\pm,0}\big),
\end{align*}
we get
\begin{align*}
\|(U_s-c)A(Y)\partial_Y^2 \chi_{\pm ,1}\|_{L^\infty_{\theta,\pm}} \lesssim
|\log c_i|.
\end{align*}

Plugging all the above estimates,  we have
\begin{equation}
\begin{aligned}
\|\chi_{\pm,1}\|_{Y_{\theta,\pm}}
\lesssim
|\log c_i|.
\label{CR34}
\end{aligned}
\end{equation}

\medskip

Next, we give estimates for $\chi_{\pm,2}$.
By \eqref{CR47}, we write
\begin{align}\label{def: chi_2}
\chi_{\pm,2}=(U_s-c)\int_{\pm \infty}^Y G_\pm(Z)dZ,
\end{align}
where
\begin{align*}
G_\pm(Y)=\frac{\chi_{\pm,1}}{U_s-c}
+  \frac{\widetilde{G}_\pm(Y)}{{(U_s-c)}^2} 
-m^2 \widetilde{G}_\pm(Y),\quad \widetilde{G}_\pm(Y)\coloneqq
 \int_{\pm \infty}^Y \frac{{(U_s-c)}^2-{(1\mp c)}^2}{\gamma_\pm^2}
dZ.
\end{align*}

By \eqref{CR34} and Lemma \ref{con:CR2}, Lemma \ref{con:CR4}, we know
\begin{equation}
\begin{aligned}
\|\chi_{\pm,2}\|_{L^\infty_{\theta,\pm}}
&\lesssim
|\log c_i|
\|\chi_{\pm,1}\|_{L^\infty_{\theta,\pm}}
+\|\widetilde{G}_\pm\|_{L^\infty_{\theta,\pm}}
+\big\|(U_s-c)\int_{\pm \infty}^Y \frac{\widetilde{G}_\pm(Z)}{{(U_s-c)}^2} dZ\big\|_{L^\infty_{\theta,\pm}}\\
&\lesssim {|\log c_i|}^2+
\big\|(U_s-c)\int_{\pm \infty}^Y \frac{\widetilde{G}_\pm(Z)}{{(U_s-c)}^2} dZ\big\|_{L^\infty_{\theta,\pm}}
.
\label{CR32}
\end{aligned}
\end{equation}
For $Y\in [Y_c+ 1,+\infty)\cup (-\infty,Y_c- 1]$, we know $|U_s-c| \approx 1$. By Lemma \ref{con:CR2}, we have
\begin{equation}
\begin{aligned}
\big|(U_s-c)\int_{\pm \infty}^Y \frac{\widetilde{G}_\pm(Z)}{{(U_s-c)}^2} dZ\big|\lesssim \|\widetilde{G}_\pm\|_{L^\infty_{\theta,\pm}}
\lesssim\big\|{(U_s-c)}^2-{(1\mp c)}^2\big\|_{L^\infty_{\theta,\pm}}
e^{\mp\theta Y}
\lesssim 
e^{\mp\theta Y}.
\label{CR33p}
\end{aligned}
\end{equation}
For $Y\in [Y_c- 1, Y_c+ 1]$, we notice $e^{\mp \theta Y}\approx 1$. We apply Lemma \ref{con:G} to have
\begin{equation}
\begin{aligned}
\big|(U_s-c)\int_{\pm \infty}^Y \frac{\widetilde{G}_\pm(Z)}{{(U_s-c)}^2} dZ\big|\lesssim |\log c_i|(\|\widetilde{G}_\pm\|_{L^\infty_{\theta,\pm}}+\|\pa_Y\widetilde{G}_\pm\|_{L^\infty_{\theta,\pm}})
\lesssim
|\log c_i|
e^{\mp\theta Y}, 
\end{aligned}
\end{equation}
Thus, we have
\begin{equation}
\begin{aligned}
\label{C7}
\| \chi_{\pm,2}\|_{L^\infty_{\theta,\pm}}
\lesssim
{|\log c_i|}^2.
\end{aligned}
\end{equation}

Taking the derivative $\pa_Y$ of \eqref{def: chi_2}, we have
\begin{align*}
	\partial_Y \chi_{\pm,2}=&\partial_Y U_s\int_{\pm \infty}^Y G_\pm(Z)dZ
+(U_s-c)G_\pm(Y)\\
=&\partial_Y U_s\int_{\pm \infty}^Z \frac{\chi_{\pm,1}(Z)}{U_s-c}
+  \frac{\widetilde{G}_\pm(Z)}{{(U_s-c)}^2} 
-m^2 \widetilde{G}_\pm(Z)dZ +\chi_{\pm,1}(Y)
+  \frac{\widetilde{G}_\pm(Y)}{U_s-c} 
-m^2 (U_s-c)\widetilde{G}_\pm(Y).
\end{align*}
By \eqref{CR34} and Lemma \ref{con:CR2}, Lemma \ref{con:CR4}, we know
\beno
\|\pa_Y\chi_{\pm,2}\|_{L^\infty_{\theta,\pm}}
&\lesssim
\big\|\partial_Y U_s\int_{\pm \infty}^Y \frac{\widetilde{G}_\pm(Z)}{{(U_s-c)}^2} dZ 
+  \frac{\widetilde{G}_\pm(Y)}{U_s-c} \big\|_{L^\infty_{\theta,\pm}}
+{|\log c_i|}^2.
\eeno
For $Y\in [Y_c+ 1,+\infty)\cup (-\infty,Y_c- 1]$, we know $|U_s-c| \approx 1$. By Lemma \ref{con:CR2}, we have
\beno
\big|\partial_Y U_s\int_{\pm \infty}^Y \frac{\widetilde{G}_\pm(Z)}{{(U_s-c)}^2} dZ 
+  \frac{\widetilde{G}_\pm(Y)}{U_s-c} \big|\lesssim
\big\|{(U_s-c)}^2-{(1\mp c)}^2\big\|_{L^\infty_{\theta,\pm}}
e^{\mp\theta Y}
\lesssim 
e^{\mp\theta Y}.
\eeno
For $Y\in[Y_c- 1, Y_c+ 1]$, we notice $e^{\mp \theta Y}\approx 1$. We apply Lemma \ref{con:G} and we have
\beno
\big|\partial_Y U_s\int_{\pm \infty}^Y \frac{\widetilde{G}_\pm(Z)}{{(U_s-c)}^2} dZ 
+  \frac{\widetilde{G}_\pm(Y)}{U_s-c}\big|\lesssim
|\log c_i|.
\eeno
Thus, we have
\begin{equation}
\begin{aligned}
\label{C8}
\|\partial_Y \chi_{\pm,2}\|_{L^\infty_{\theta,\pm}}
\lesssim
{|\log c_i|}^2.
\end{aligned}
\end{equation}

By \eqref{CR4}, \eqref{C7}, \eqref{C8} and the equation of $\chi_{\pm,2}$:
\begin{align*}
{Ray}_0[\chi_{\pm,2}]=(U_s-c)\big( 2A^{-1}\partial_Y\chi_{\pm,1}+\partial_Y(A^{-1})\chi_{\pm,1}\big)- (U_s-c)\big(A^{-1}-\gamma_\pm^{-2} \big)\chi_{\pm,0},
\end{align*}
 we get
 \begin{align*}
 \|(U_s-c)A\pa_Y^2 \chi_{\pm,2}\|_{L^\infty_{\theta,\pm}} \lesssim {|\log c_i|}.
 \end{align*}

Combining all the above estimates, we get
\begin{equation*}
\begin{aligned}
\|\chi_{\pm,2}\|_{Y_{\theta,\pm}}
\lesssim
|\log c_i|^2,
\end{aligned}
\end{equation*}
  which finishes the proof of Lemma \ref{con:CR6}.

\bigskip

\section{The Dispersive relation}

In this section, we solve the dispersion relation $W(\alpha,c)=0$. To get the dispersion relation, we define
\beno
\mathbb{G}= \left\{
\alpha\in\mathbb R:0<\alpha\ll 1
\right\}.
\eeno

\subsection{The case $m<\sqrt{2}$}

We define
\begin{align*}
&\mathbb{H}_<=\left\{
(\alpha,c)\in\mathbb G\times\mathbb C:
\left|c-\mathrm{i}\sigma\right|
\leq S\alpha
\right\},\\
&\mathbb{H}_{<,\alpha}\coloneqq \big\{
c\in \mathbb C||c- \sigma  i|\leq S\alpha
\big\},
\end{align*}
where $\sigma:=\sigma(m)=\sqrt{\frac{ \sqrt{4m^2+1}-(m^2+1)}{m^2}}>0$ and $S$ is a constant to be chosen later.

We get that for a given $\al\in\mathbb{G}$,
if $c\in \mathbb{H}_{<,\alpha}$, $(\alpha,c)\in \mathbb{H}_<$. The dispersion relation \eqref{equ: dispersion} is equivalent to
\begin{align}\label{dispersion 1}
m^2 {(1-c^2)}^2-2(c^2+1)=\mathcal{R}(\alpha,c).
\end{align}
By the construction of the solution and the differentiability of parameters for ODE, we know $\mathcal{R}(\alpha,c)$ is analytic in $(\alpha,c)$ near $(\alpha,c)=(0, \sigma i)$. Moreover, for any $(\alpha,c)\in \mathbb{H}_<$, we have
\begin{equation}\label{est: R}
\begin{aligned}
&|\mathcal{R}(\alpha,c)|\lesssim \alpha,\quad |\partial_{c_r}\mathcal{R}(\alpha,c)|
+|\partial_{c_i}\mathcal{R}(\alpha,c)|
\lesssim \alpha.
\end{aligned}
\end{equation}

We shall prove the following proposition:
\begin{proposition}\label{con:dis}
Under the assumptions of Theorem \ref{thm:main1}, there exists $S>0$ such that for every $\alpha\in\mathbb{G}$, there is a corresponding $c(\alpha)$ for which the pair $(\alpha,c(\alpha))\in \mathbb{H}_<$ satisfies the dispersion relation \eqref{dispersion 1}. Moreover, $(\alpha,c)$ satisfies the following properties:
\begin{align*}
0<\alpha\ll 1, \quad  c_r=O(\al),\quad c_i=\sigma+O(\al)>0.
\end{align*}
\end{proposition}

\medskip

To prove the existence of $c(\al)\in\mathbb{H}_{<,\alpha}$ to the dispersion relation \eqref{dispersion 1}, we introduce an operator $\mathbb{T}_<:\mathbb{H}_{<,\alpha}\rightarrow \mathbb{C}$ defined by
\begin{equation*}
\begin{aligned}
m^2 {(1-{({\mathbb{T}_<u})}^2)}^2-2({({\mathbb{T}_<u})}^2+1)=\mathcal{R}(\alpha,u),
\end{aligned}
\end{equation*}
which is equivalent to
\begin{equation} \label{CR51}
\begin{aligned}
{({\mathbb{T}_<u})}^2=\frac{m^2+1- \sqrt{\big(4+\mathcal{R}(\alpha,u)\big)m^2+1}}{m^2} =-\sigma^2+O(\alpha).
\end{aligned}
\end{equation}
Take
\begin{equation}\label{CD1}
\begin{aligned}
{\mathbb{T}_<u}=\sigma i+O(\alpha),
\end{aligned}
\end{equation}
where  we require $\operatorname{Im}{(\mathbb{T}_<u)}>0$.

\medskip

\begin{lemma}\label{con:fix1}
    The operator $\mathbb{T}_<$ is a contraction map and satisfies 
\beno
|{\mathbb{T}_<u}-\mathbb{T}_<v|\lesssim \alpha |u-v|, \quad\forall u,v\in \mathbb{H}_{<,\alpha}.
\eeno
\end{lemma}

\begin{proof}
    
For all $ u,v\in \mathbb{H}_{<,\alpha}$, by \eqref{CR51}, we have
\begin{equation}
\begin{aligned}
({\mathbb{T}_<u}+\mathbb{T}_<v)({\mathbb{T}_<u}-\mathbb{T}_<v)&=
\frac{\sqrt{\big(4+\mathcal{R}(\alpha,v)\big)m^2+1}- \sqrt{\big(4+\mathcal{R}(\alpha,u)\big)m^2+1}}{m^2}\\
&=\frac{\mathcal{R}(\alpha,v)-\mathcal{R}(\alpha,u)}{\sqrt{\big(4+\mathcal{R}(\alpha,v)\big)m^2+1}+ \sqrt{\big(4+\mathcal{R}(\alpha,u)\big)m^2+1}}
.
\label{D13}
\end{aligned}
\end{equation}
By noticing that
\begin{equation*}
\begin{aligned}
&\big|\mathcal{R}(\alpha,u)-\mathcal{R}(\alpha,v)\big|\\
=& \big|\int_0^1 \frac{d}{ds}
\mathcal{R}(\alpha,su+(1-s)v)ds\big|\\
\lesssim &
\int_0^1 |u-v| \big(\big| \partial_{c_r}
\mathcal{R}(\alpha,su+(1-s)v)\big|
+\big| \partial_{c_i}
\mathcal{R}(\alpha,su+(1-s)v)\big|\big)ds\\
\lesssim & \alpha |u-v|,
\end{aligned}
\end{equation*}
where we used \eqref{est: R} in the last step.

Moreover, by \eqref{CD1}, we have
\beno
|{\mathbb{T}_<u}+\mathbb{T}_<v|\approx 1, \quad
\big|\sqrt{\big(4+\mathcal{R}(\alpha,v)\big)m^2+1}+ \sqrt{\big(4+\mathcal{R}(\alpha,u)\big)m^2+1}\big|\approx 1,
\eeno
which combines with \eqref{D13} to obtain
\beno
|{\mathbb{T}_<u}-\mathbb{T}_<v|\lesssim \alpha |u-v|,
\eeno
which implies that $\mathbb{T}_<$ is a contraction map.

\end{proof}

\medskip

Now, we are in a position to prove Proposition \ref{con:dis}.

\medskip

\underline{\it Proof of Proposition \ref{con:dis}.}
Denote $c_{(0)}\coloneqq \sigma i.$ It is easy to check $(\alpha,c_{(0)})\in \mathbb{H}_<$.  

Define 
\beno
c_{(j+1)}\coloneqq \mathbb{T}_<c_{(j)} \in \mathbb{C},~j\geq 0.
\eeno
Using Lemma \ref{con:fix1} to get that for any $j\geq 0$, it holds that
\begin{align}\label{est: c_j}
|c_{(j+1)}-c_{(0)}|\lesssim \sum_{k=0}^{j}|{\mathbb{T}_<^{k+1}c_{(0)}}-\mathbb{T}_<^{k}c_{(0)}|
\lesssim  \sum_{k=0}^{j} \alpha^k |c_{(1)}-c_{(0)}|
\lesssim   |c_{(1)}-c_{(0)}| \leq S \alpha,
\end{align}
where $S>0$ is a constant independent of $j$. 

The inequality \eqref{est: c_j} means $c_{(j+1)} \in \mathbb{H}_{<,\alpha}$ for any $j\geq 0$ . Using Lemma \ref{con:fix1} to get
\begin{align*}
	&|{c_{(j+1)}}-c_{(j)}|\lesssim \alpha |c_{(j)}-c_{(j-1)}|,
\end{align*}
which means the sequence $\{c_{(j)}\}_{j=0}^{\infty}$ is convergent. 

Define
\beno
c=c(\al)\coloneqq {\lim_{j \to \infty}}c_{(j)} \in \mathbb{H}_{<,\alpha},
\eeno
 which solves the dispersion relation \eqref{dispersion 1}. According to \eqref{dispersion 1} and \eqref{est: R}, we have
 \begin{align*}
 c^2=-\sigma^2+O(\al),
 \end{align*}
which implies that
\begin{align*}
 c_r=O(\al),\quad c_i=\sigma+O(\al)>0.
\end{align*}
Here, we finish the proof.

\qed

\bigskip

\subsection{The dispersion relation when $m\geq \sqrt{2}$}
For $m\geq \sqrt 2$, the terms of order $O(\al)$ vanish. So we should compute the terms of order $O(\al^2)$ carefully.

Recall
\begin{equation}
\begin{aligned}
c_*:=c_*(m)=\sqrt{\frac{m^2+1- \sqrt{4m^2+1}}{m^2} }, 
\end{aligned}
\end{equation}
which satisfies
\begin{equation}
\begin{aligned}
\label{U1}
m^2 {(1-c_*^2)}^2-2(c_*^2+1)=0,
\end{aligned}
\end{equation}
and
\begin{equation}
\begin{aligned}
&c_*=0, \quad \text{when} \ m=\sqrt{2},\\
&c_*\in (0,1), \quad \text{when} \ m>\sqrt{2}.
\end{aligned}
\end{equation}

By Taylor expansion at $c=c_*$, together with \eqref{U1}, we directly derive the following lemma.

\begin{lemma}
\label{dis 2}
When $m>\sqrt{2}$, we have
  $$  m^2 {(1-c^2)}^2-2(c^2+1)=-4c_*\sqrt{4m^2+1}(c-c_*) +O(|c-c_*)|^2) .$$
  When $m=\sqrt{2}$, we have
  $$  m^2 {(1-c^2)}^2-2(c^2+1)=-6c^2 +O(c^4) .$$
\end{lemma}

Next, we calculate order $O(\al^2)$ carefully.  For $m=\sqrt 2$, we define
\begin{align*}
&\mathbb{H}_=\left\{
(\alpha,c)\in\mathbb G\times\mathbb C:
\left|
c-e^{\pi\mathrm{i}/6}\delta
\right|
\leq
S\delta^2|\log\delta|^2
\right\},\\
&\mathbb{H}_{=,\alpha}\coloneqq \big\{
c\in \mathbb C \Big| |c-
e^{\frac{\pi i}{6}}
\delta|\leq S\delta^2  |\log \delta|^2
\big\},
\end{align*}
where $\delta={(\alpha/3)}^{1/3}>0$ and $S>0$ is a constant to be chosen later.

For $m>\sqrt 2$, we define
\begin{align*}
&\mathbb{H_{>}}\coloneqq \big\{
(\alpha,c)| \alpha \ll 1,|c-(c_*+K_m\alpha)|\leq S\alpha^2|\log \al|^6
\big\},\\
&\mathbb{H_{>,\alpha}}\coloneqq \Big\{c\big|
|c-(c_*+K_m\alpha)|\leq S\alpha^2|\log \al|^6
\big\},
\end{align*}
where 
\begin{align*}
c_*=&\sqrt{\frac{m^2+1- \sqrt{4m^2+1}}{m^2} }\in (0,1),\\
K_m=& \frac{1-c_*^2}{8c_*^2\sqrt{4m^2+1}} \left[ 2\pi c_*+ i\left(4+2c_*\log\frac{1+c_*}{1-c_*}\right) \right],\\
\operatorname{Im} K_m&>0,\quad \operatorname{Re} K_m>0,
\end{align*}
and $S>0$ is a constant to be chosen later.

We denote
\begin{align*}
\xi_{m, U_s}=-\lim_{\e\to 0^+}\mathcal I_{m,U_s}(c)\Big|_{c=c_*+i \e},
\end{align*}
where $\mathcal I_{m,U_s}(c)$ is given in \eqref{def: ImU_s}.	
\medskip

 
Next, we give the precise calculation of  $\xi_{m, U_s}$.

\begin{lemma}
\label{con:xi}
For $U_s(Y)=\tanh (Y)$, we have \begin{align*}
	\xi_m:=\xi_{m, \tanh Y} = 4
 + 2c_* \left(\log\left(\frac{1+c_*}{1-c_*}\right) -\pi i  \right).
\end{align*}
\end{lemma}
\begin{proof}
    We have the following formulas:
\begin{equation}
\begin{aligned}
\partial_Y U_s=1-U_s^2,\quad
\partial_Y^2 U_s=-2U_s\partial_Y U_s.
\label{CR41r}
\end{aligned}
\end{equation}

Setting $Z=U_s(Y)$, we obtain
\begin{align*}
	\xi_m = -
 \lim_{\e\to 0^+}
 \bigg( & \int_0^{1} \frac{\big({(Z-c)}^2-{(1- c)}^2\big)
\big( 
{(Z-c)}^2
-\frac{{(1+c)}^2}{\gamma_+ \gamma_-}
\big)
}{{(Z-c)}^2 (1-Z^2)}
dZ\\
&\quad\quad  +  \int_{- 1}^0 \frac{\big({(Z-c)}^2-{(1+ c)}^2\big)
\big( 
{(Z-c)}^2-
\frac{{(1-c)}^2}{\gamma_+ \gamma_-}
\big)
}{{(Z-c)}^2(1-Z^2)}
dZ \bigg)\Big|_{c=c_*+i \e}.
\end{align*}
We decompose $\xi_m$ as $\xi_m = \xi_m^1 + \xi_m^2$:
\begin{align*}
	\xi_m^1 = -
\lim_{\e\to 0^+}
 \bigg( & \int_{- 1}^{1} \frac{\big({(Z-c)}^2-{(1- c)}^2\big)
\big( 
{(Z-c)}^2
-{(1+c)}^2
\big)
}{{(Z-c)}^2 (1-Z^2)}
dZ\Big|_{c=c_*+i \e},\\
	\xi_m^2 = -
 \lim_{\e\to 0^+}
 \bigg( & {(1+c)}^2\big( 
1
-\frac{1}{\gamma_+ \gamma_-}
\big)\int_0^{1} \frac{{(Z-c)}^2-{(1- c)}^2
}{{(Z-c)}^2 (1-Z^2)}
dZ\\
&\quad\quad  + {(1-c)}^2\big( 
1
-\frac{1}{\gamma_+ \gamma_-}
\big) \int_{- 1}^0 \frac{{(Z-c)}^2-{(1+ c)}^2
}{{(Z-c)}^2(1-Z^2)}
dZ \bigg)\Big|_{c=c_*+i \e}.
\end{align*}

For $\xi_m^2$, we first compute the two integrals explicitly.
A direct  decomposition gives
\begin{align*}
\frac{(Z-c)^2-(1-c)^2}
{(Z-c)^2(1-Z^2)}
=&
-\frac{2c}{(1+c)^2}\frac{1}{Z-c}
-\frac{1-c}{1+c}\frac{1}{(Z-c)^2}+\frac{2c}{(1+c)^2}\frac{1}{Z+1}.
\end{align*}
It follows that
\begin{align*}
&\int_0^1
\frac{(Z-c)^2-(1-c)^2}
{(Z-c)^2(1-Z^2)}
\,dZ
=
\frac{1}{c(1+c)}
+
\frac{2c}{(1+c)^2}
\left(
\log 2+\log(-c)-\log(1-c)
\right).
\end{align*}
Similarly, we have
\begin{align*}
&\int_{-1}^0
\frac{(Z-c)^2-(1+c)^2}
{(Z-c)^2(1-Z^2)}
\,dZ\\
=&
-\frac{1}{c(1-c)}
-
\frac{2c}{(1-c)^2}
\left(
\log 2+\log c-\log(1+c)
\right).
\end{align*}
Therefore,
\begin{align*}
& (1+c)^2
\int_0^1
\frac{(Z-c)^2-(1-c)^2}
{(Z-c)^2(1-Z^2)}
\,dZ
+(1-c)^2
\int_{-1}^0
\frac{(Z-c)^2-(1+c)^2}
{(Z-c)^2(1-Z^2)}
\,dZ\\
={}&
\frac{1+c}{c}
-\frac{1-c}{c}
+2c\left(
\log 2+\log(-c)-\log(1-c)
\right)
-2c\left(
\log 2+\log c-\log(1+c)
\right)\\
={}&
2+2c\left(
\log(-c)-\log c
+\log(1+c)-\log(1-c)
\right).
\end{align*}
Here and below, $\log$ denotes the principal branch of the
complex logarithm. Since $c_i>0$, the principal branch gives
$$
\log(-c)-\log c=-\pi i.
$$
Consequently,
\begin{align*}
& (1+c)^2
\int_0^1
\frac{(Z-c)^2-(1-c)^2}
{(Z-c)^2(1-Z^2)}
\,dZ
+(1-c)^2
\int_{-1}^0
\frac{(Z-c)^2-(1+c)^2}
{(Z-c)^2(1-Z^2)}
\,dZ\\
=&
2+2c\left(
\log(1+c)-\log(1-c)-\pi i
\right).
\end{align*}
Since $c_*\in[0,1)$, the right-hand side is uniformly bounded
as $c\to c_*+i0^+$. On the other hand, by
\eqref{U2},
\[
\gamma_+\gamma_-
=
1+O(|c-c_*|),
\]
we obtain
\begin{align*}
&
\left|
\left(
1-\frac{1}{\gamma_+\gamma_-}
\right)
\left[
2+2c\left(
\log(1+c)-\log(1-c)-\pi\mathrm{i}
\right)
\right]
\right|\lesssim |c-c_*|
\longrightarrow 0
\qquad
\text{as }c\to c_*+i0^+.
\end{align*}
It follows from the definition of $\xi^2_m$ that
$$
\xi^2_m=0.
$$

For $\xi^1_m$, notice that
\begin{align*}
	- \frac{\big({(Z-c)}^2-{(1- c)}^2\big)
\big( 
{(Z-c)}^2
-{(1+c)}^2
\big)
}{{(Z-c)}^2 (1-Z^2)}=
1 - \frac{2c}{Z-c} + \frac{c^2-1}{(Z-c)^2}.
\end{align*}
So we have
\begin{align*}
	\xi^1_m=& 
  \lim_{\e\to 0^+}
 \bigg(  \int_{- 1}^{1} 1 - \frac{2c}{Z-c} + \frac{c^2-1}{(Z-c)^2}
dZ \bigg)\Big|_{c=c_*+i \e}\\
=& 
\lim_{\e\to 0^+}
\left(  4 - 2c \left(\log (1-c)-\log (-1-c)\right)
\right)\Big|_{c=c_*+i \e},
\end{align*}
where $\log$ denotes the principal branch of the complex logarithm with $\arg z\in (-\pi,\pi)$, which satisfies
$$\log z =\log |z| +i\arg z,\quad \arg z\in (-\pi,\pi).$$
So we have
\begin{align*}
 \lim_{\e\to 0^+}
 \log (1-c)|_{c=c_*+i \e}=&\log (1-c_*),\\
  \lim_{\e\to 0^+}
 \log (-1-c)|_{c=c_*+i \e}=&\log (1+c_*)-\pi i.
\end{align*}
So we have
\begin{align*}
	\xi_m^1 = 4
 + 2c_* \left(\log\left(\frac{1+c_*}{1-c_*}\right) -\pi i  \right).
\end{align*}
So we have
\begin{align*}
	\xi_m =\xi_m^1+\xi_m^2= 4
 + 2c_* \left(\log\left(\frac{1+c_*}{1-c_*}\right) -\pi i  \right).
\end{align*}
Here, we finish the proof.
\end{proof}

\medskip

\subsection{The case $m=\sqrt{2}$}
 For the case $m=\sqrt{2}$, by Lemma \ref{dis 1}, Lemma \ref{dis 2} and Lemma \ref{con:xi}, we get  that the dispersion relation \eqref{equ: dispersion} is equivalent to
\begin{equation}
\begin{aligned}
& c^3
  -
\frac{1}{12}i \alpha 
 \big(  \xi + \mathcal{R}_1(c)
 \big)
 +\mathcal{R}_2(\alpha, c)
=0,
\label{CR55}
\end{aligned}
\end{equation}
where $\mathcal{R}_1(c)$ is analytic in $c$ when $c$ is close to 0, $\mathcal{R}_2(\alpha,c)$ is analytic in $(\alpha,c)$ near $(0,0)$. Moreover, when $(\alpha,c)$ near $(0,0)$, we also know $\mathcal{R}_1(c),\mathcal{R}_2(\alpha,c)$ are power series for $\alpha,c,\log (-c)$. 

Thus, for any $(\alpha,c)\in \mathbb{H}_=$, we have
\begin{equation}\label{est: R_1}
\begin{aligned}
&|\mathcal{R}_1(c)|\lesssim  |c||\log c_i|,\quad |\partial_{c_r}\mathcal{R}_1(c)|
+|\partial_{c_i}\mathcal{R}_1(c)|
\lesssim |\log c_i|,\\
&|\mathcal{R}_2(\alpha,c)|\lesssim \alpha^2|\log c_i|^6,\quad |\partial_{c_r}\mathcal{R}_2(\alpha,c)|
+|\partial_{c_i}\mathcal{R}_2(\alpha,c)|
\lesssim \alpha^2 \frac{|\log c_i|^5}{|c|} .
\end{aligned}
\end{equation}

\begin{proposition}\label{con:dispm}
For $m=\sqrt 2$ and $U_s(Y)=\tanh Y$, there exists a constant $S>0$ such that for every $\alpha\in\mathbb{G}$, there exists a corresponding 
$c(\alpha)$ with the following properties:
  \begin{enumerate}
\item The pair $(\alpha,c(\alpha))\in \mathbb{H}_=$ satisfies the dispersion relation \eqref{dispersion 1}.
\item The pair $(\alpha,c(\alpha))$ satisfies 
\beno
0<\alpha\ll 1, \quad c_r(\alpha)=\frac{\sqrt{3}}{2}\delta+O(\delta^2  |\log \delta|^2), \quad  c_i(\alpha)=\frac{1}{2}\delta+O(\delta^2  |\log \delta|^2)>0,
\eeno 
where $\d={(\alpha/3)}^{1/3}>0.$
\end{enumerate}
 
\end{proposition}

To prove the existence of $c(\al)$ for every $\al\in \mathbb{G} $, we introduce an operator $\mathbb{T}_{=}:\mathbb{H}_{=,\alpha}\rightarrow \mathbb{C}$,
\begin{equation}
\begin{aligned}
{(\mathbb{T}_{=} u)}^3
  -
\frac{1}{12}i \alpha 
 \big(  \xi + \mathcal{R}_1(u)
 \big)
 +\mathcal{R}_2(\alpha, u)
=0,
\label{CR52}
\end{aligned}
\end{equation}
which is equivalent to
\begin{equation}
\begin{aligned}
 {(\mathbb{T}_{=} u)}^3
  =\alpha\big( \frac{1}{12}i 
  \xi 
+O( |c||\log c_i|^6) \big)
.
\label{CR53}
\end{aligned}
\end{equation}
Take 
\begin{equation}
\begin{aligned}
{\mathbb{T}_{=}u}=e^{\frac{\pi i}{6}}
\delta+O(\delta |c|^\f13  |\log  c_i|^2),\quad \delta={(\alpha/3)}^{1/3}>0,
\label{CD1p}
\end{aligned}
\end{equation}
where we require $\operatorname{Im}{(\mathbb{T}_{=}u)}>0$.

\begin{lemma}\label{con:m1}
    The operator $\mathbb{T}_{=}$ is a contraction map and satisfies that    
    \beno
    |{\mathbb{T}_{=}u}-\mathbb{T}_{=}v|\lesssim \delta |\log  \d|^5|u-v|, \quad \forall u,v\in \mathbb{H}_{=,\alpha}.
    \eeno
\end{lemma}

\begin{proof}
    
For all $ u,v\in \mathbb{H}_{=,\alpha}$, by \eqref{CR52}, we have
\begin{equation}
\begin{aligned}
&({\mathbb{T}_{=}u}-\mathbb{T}_{=}v)
({\mathbb{T}_{=}u}-e^{\frac{2\pi i}{3}}\mathbb{T}_{=}v)
({\mathbb{T}_{=}u}-e^{\frac{4\pi i}{3}}\mathbb{T}_{=}v)\\=&
\frac{1}{12}i \alpha 
 \big(  \mathcal{R}_1(u)-\mathcal{R}_1(v)
 \big)
 +\mathcal{R}_2(\alpha, u)-\mathcal{R}_2(\alpha, v).
\label{D13p}
\end{aligned}
\end{equation}
By \eqref{est: R_1} and $|c_r|\approx c_i\approx \al \approx\d $ for $(\al, c)\in \mathbb{H}_=$, we have
\begin{equation}
\begin{aligned}
&\big|\mathcal{R}_2(\alpha,u)-\mathcal{R}_2(\alpha,v)\big|\\
=& \big|\int_0^1 \frac{d}{ds}
\mathcal{R}_2(\alpha,su+(1-s)v)ds\big|\\
\lesssim &
\int_0^1 |u-v| \big(\big| \partial_{c_r}
\mathcal{R}_2(\alpha,su+(1-s)v)\big|
+\big| \partial_{c_i}
\mathcal{R}_2(\alpha,su+(1-s)v)\big|\big)ds\\
\lesssim & \alpha^2\frac{|\log c_i|^5}{|c|}  |u-v|
\lesssim  \delta |\log \d|^5 |u-v|,
\end{aligned}
\end{equation}
and
\begin{equation}\label{est: R1-1}
\begin{aligned}
\big|\mathcal{R}_1(u)-\mathcal{R}_1(v)\big|
=& \big|\int_0^1 \frac{d}{ds}
\mathcal{R}_1(su+(1-s)v)ds\big|\\
\lesssim &
\int_0^1 |u-v| \big(\big| \partial_{c_r}
\mathcal{R}_1(su+(1-s)v)\big|
+\big| \partial_{c_i}
\mathcal{R}_1(su+(1-s)v)\big|\big)ds\\
\lesssim & |\log \delta| |u-v|.
\end{aligned}
\end{equation}
According to \eqref{CD1p}, we get
$$\big|{\mathbb{T}_{=}u}-e^{\frac{2\pi i}{3}}\mathbb{T}_{=}v\big|\approx \delta,\quad \  \big|{\mathbb{T}_{=}u}-e^{\frac{4\pi i}{3}}\mathbb{T}_{=}v\big|\approx \delta.$$
From \eqref{D13p}-\eqref{est: R1-1}, we get
\beno
|{\mathbb{T}_{=}u}-\mathbb{T}_{=}v|\lesssim \delta |\log \delta|^5|u-v|,
\eeno
which implies that $\mathbb{T}_{=}$ is a contraction map.
\end{proof}

\medskip

Now, we are in a position to prove Proposition \ref{con:dispm}.

\underline{\it Proof of Proposition \ref{con:dispm}.}
Denote $c_{(0)}\coloneqq e^{\frac{\pi i}{6}}
\delta .$ It is easy to check $(\alpha,c_{(0)})\in \mathbb{H}_=$.  

Define 
\beno
c_{(j+1)}\coloneqq \mathbb{T}_{=}c_{(j)} \in \mathbb{C}, j\geq 0.
\eeno
By \eqref{CR53} and Lemma \ref{con:m1}, we obtain that for any $j\geq 0$, it holds
\begin{align*}
	|c_{(j+1)}-c_{(0)}|&\lesssim \sum_{k=0}^{j}|{\mathbb{T}_=^{k+1}c_{(0)}}-\mathbb{T}_=^{k}c_{(0)}|
\lesssim  \sum_{k=0}^{j} (\delta |\log \delta|^5)^k |c_{(1)}-c_{(0)}|\\
&\lesssim   |c_{(1)}-c_{(0)}| =|\mathbb T c_{(0)}-e^{\frac{\pi i}{6}}
\delta |\leq C\delta^2  |\log \delta|^5,
\end{align*}
where $S>0$ is a constant independent of $j$. This implies that $c_{(j+1)}\in \mathbb{H}_{=,\alpha}$ for any $j\geq 0$.

Moreover, by Lemma \ref{con:m1}, we get
\begin{align*}
	&|{c_{(j+1)}}-c_{(j)}|\lesssim \delta |\log \delta|^5 |c_{(j)}-c_{(j-1)}|, \quad \delta |\log \delta|^5\ll 1,
\end{align*}
which means the sequence $\{c_{(j)}\}_{j=0}^{\infty}$ is convergent. 

Define
$$c\coloneqq {\lim_{j \to \infty}}c_{(j)} \in \mathbb{H}_{=,\alpha},$$
which solves the dispersion relation \eqref{CR55}. According to \eqref{CR55} and \eqref{est: R_1}, we have
 \begin{align*}
 c^3=\f{\al i}{3}+O(\al |c||\log c_i|),
 \end{align*}
which implies that $c=e^{\f{\pi}{6}i}\d+O(\d |c||\log c_i|),\quad \d={(\alpha/3)}^{1/3}>0$. That is 
\begin{align*}
c_r=\frac{\sqrt{3}}{2}\delta+O(\delta^2  |\log \delta|^2), \quad  c_i=\frac{1}{2}\delta+O(\delta^2  |\log \delta|^2)>0.
\end{align*}

Here, we finish the proof.

\subsection{The case $m>\sqrt{2}$}
 For $m>\sqrt{2}$, 
 by Lemma \ref{dis 1}, Lemma \ref{dis 2} and Lemma \ref{con:xi}, we get  that the dispersion relation \eqref{equ: dispersion} is equivalent to
\begin{equation}
\begin{aligned}
\frac{-8i c_*^2\sqrt{4m^2+1}}{1-c_*^2}(c-c_*)\alpha - 
\left(
4+2c_* \left(\log\left(\frac{1+c_*}{1-c_*}\right) -\pi i  \right)  
 \right)\alpha^2 +O(\alpha^3|\log c_i|^6)=0.
\end{aligned}
\end{equation}
Set
 $$c=c_*+\kappa \alpha.$$
We get  that the dispersion relation \eqref{equ: dispersion} is equivalent to
\begin{equation}
\begin{aligned}
\label{CR55t}
\kappa=K_m
+\mathcal{R}(\alpha, c)
,
\end{aligned}
\end{equation}
where 
\begin{equation}
\begin{aligned}
K_m= \frac{1-c_*^2}{8c_*^2\sqrt{4m^2+1}} \left[ 2\pi c_*+ i\left(4+2c_*\log\frac{1+c_*}{1-c_*}\right) \right]
\end{aligned}
\end{equation}
is a constant depending only on $m$ and satisfies $\operatorname{Im}K_m>0,~\operatorname{Re} K_m>0$.  $\mathcal{R}(\alpha,c)$ is analytic in $(\alpha,c)$ when $\alpha$ near $0$ and $c$ near $c_*$. Moreover, when $(\alpha,c)$ near $(0,c_*)$, we also know $\mathcal{R}(\alpha,c)$ are power series for $\alpha,c,\log (-(c-c_*))$.

Thus, for any $(\alpha,c)\in \mathbb{H}_>$, we have
\begin{equation}\label{est: R_1t}
\begin{aligned}
&|\mathcal{R}(\alpha,c)|\lesssim \alpha|\log c_i|^6,\quad |\partial_{c_r}\mathcal{R}(\alpha,c)|
+|\partial_{c_i}\mathcal{R}(\alpha,c)|
\lesssim \alpha |\log \alpha|^5 .
\end{aligned}
\end{equation}

\begin{proposition}\label{con:dispmt}
For $m>\sqrt 2$ and $U_s(Y)=\tanh Y$, there exists a constant $S>0$ such that for every $\alpha\in\mathbb{G}$, there exists a corresponding 
$c(\alpha)$ with the following properties:
  \begin{enumerate}
\item The pair $(\alpha,c(\alpha))\in \mathbb{H}_>$ satisfies the dispersion relation \eqref{dispersion 1}.
\item The pair $(\alpha,c(\alpha))$ satisfies 
\beno
0<\alpha\ll 1, \quad c_r=c_*+\operatorname{Re} K_m\alpha+O(\alpha^2)>0, \quad  c_i=\operatorname{Im} K_m\alpha+O(\alpha^2)>0.
\eeno 
\end{enumerate}
 
\end{proposition}

\medskip

To prove the existence of $c(\al)$ for every $\al\in \mathbb{G} $, we introduce an operator $\mathbb{T}_{>}:\mathbb{H}_{>,\alpha}\rightarrow \mathbb{C}$,
\begin{equation}
\begin{aligned}
\mathbb{T}_{>} u 
=c_*+\alpha 
\left(K_m
+\mathcal{R}(\alpha, u)\right),
\label{CR52t}
\end{aligned}
\end{equation}
which implies
$$|\mathbb{T}_{>} u -(c_*+K_m\alpha)|
\leq \alpha|\mathcal{R}(\alpha, u)|
\leq S\alpha^2|\log c_i|^6,
$$
where $S>0$ is a constant. So we know  $\mathbb{T}_{>}$ is an operator from $\mathbb{H}_{>,\alpha}$ to $\mathbb{H}_{>,\alpha}$.

\begin{lemma}\label{con:m1t}
    The operator $\mathbb{T}_{>}$ is a contraction map and satisfies     
    \beno
   |{\mathbb{T}_{>}u}-\mathbb{T}_{>}v|\lesssim \alpha^2 |\log \alpha|^5|u-v|,\quad \forall u,v\in \mathbb{H}_{>,\alpha}.
    \eeno
\end{lemma}

\begin{proof}
    
For all $ u,v\in \mathbb{H}_{>,\alpha}$, by \eqref{CR52t}, we have
\begin{equation}
\begin{aligned}
&({\mathbb{T}_{>}u}-\mathbb{T}_{>}v)
=\alpha \left(
\mathcal{R}(\alpha, u)-\mathcal{R}(\alpha, v)
\right).
\label{D13pt}
\end{aligned}
\end{equation}
By \eqref{est: R_1t}, we have
\begin{equation}
\begin{aligned}
\label{est: R1-1t}
&\big|\mathcal{R}(\alpha,u)-\mathcal{R}(\alpha,v)\big|\\
=& \big|\int_0^1 \frac{d}{ds}
\mathcal{R}(\alpha,su+(1-s)v)ds\big|\\
\lesssim &
\int_0^1 |u-v| \big(\big| \partial_{c_r}
\mathcal{R}(\alpha,su+(1-s)v)\big|
+\big| \partial_{c_i}
\mathcal{R}(\alpha,su+(1-s)v)\big|\big)ds\\
\lesssim & \alpha |\log \alpha |^5 |u-v|.
\end{aligned}
\end{equation}
From \eqref{D13pt} and \eqref{est: R1-1t}, we get
\beno
|{\mathbb{T}_{>}u}-\mathbb{T}_{>}v|\lesssim \alpha^2 |\log \alpha|^5|u-v|,
\eeno
which implies that $\mathbb{T}_{>}$ is a contraction map.
\end{proof}

\medskip

Now, we are in a position to prove Proposition \ref{con:dispmt}.

\underline{\it Proof of Proposition \ref{con:dispmt}.}
Denote $c_{(0)}\coloneqq c_*+\alpha 
K_m .$ It is easy to check $(\alpha,c_{(0)})\in \mathbb{H}_>$.  

Define 
\beno
c_{(j+1)}\coloneqq \mathbb{T}_{>}c_{(j)} \in \mathbb{H}_{>,\alpha}, j\geq 0.
\eeno
 This implies that $c_{(j+1)}\in \mathbb{H}_{>,\alpha}$ for any $j\geq 0$.

Moreover, by Lemma \ref{con:m1t}, we get
\begin{align*}
	&|{c_{(j+1)}}-c_{(j)}|\lesssim \alpha^2 |\log \alpha|^5 |c_{(j)}-c_{(j-1)}|, \quad \alpha^2 |\log \alpha|^5\ll 1,
\end{align*}
which means the sequence $\{c_{(j)}\}_{j=0}^{\infty}$ is convergent. 

Define
$$c\coloneqq {\lim_{j \to \infty}}c_{(j)} \in \mathbb{H}_{>,\alpha},$$
which solves the dispersion relation \eqref{CR55t}. Moreover, we have
 \begin{align*}
c=c_*+K_m\alpha+O(\alpha^2),
 \end{align*}
which implies that 
\begin{align*}
c_r=c_*+\operatorname{Re} K_m\alpha+O(\alpha^2)>0, \quad  c_i=\operatorname{Im} K_m\alpha+O(\alpha^2)>0.
\end{align*}

Here, we finish the proof.
 
\qed

Finally, we present the proof of Theorem \ref{thm:main1}.

\underline{\it Proof of Theorem \ref{thm:main1}.}
Propositions 6.1, 6.5, and 6.7 give, in the three respective cases,
$c=c(\alpha)$ satisfying the dispersion relation
\begin{align*}
W(\alpha,c(\alpha))=0,
\end{align*}
and the asymptotic formulas stated in the theorem. In particular,
$c_i(\alpha)>0$ for sufficiently small $\alpha>0$. 
Due to 
\begin{align*}
\rho
&=m^2A^{-1}
\bigl[U_s'\varphi-(U_s-c)\varphi'\bigr],\\
u
&=\varphi'-(U_s-c)\rho,\quad v
=-\mathrm{i}\alpha\varphi,
\end{align*}
we use the structural assumptions on $U_s$ and
Proposition~\ref{pro: varphi} to obtain
\begin{align*}
\rho,u&\in W^{1,\infty}(\mathbb R),
\qquad
v\in W^{2,\infty}(\mathbb R).
\end{align*}
These functions solve the system (1.3) and satisfy the boundary condition (1.4).
\qed

%
%
%

\section*{Acknowledgments}
C. Wang is partially supported by the NSF of China under Grant 12471189. Y. Wang is
partially supported by the NSF of China under Grant 12471200. Z. Zhang is partially supported by NSF of China under 12288101.

\end{document}